\documentclass[preprint,12pt]{elsarticle}
\usepackage[margin=1.5cm]{geometry}
\usepackage[all]{xy}
\usepackage{amssymb} 
\usepackage{amsmath,amsthm} 
\usepackage{graphicx} 
\usepackage{ifthen}
\usepackage{pgfmath}
\usepackage[usenames, dvipsnames]{xcolor}
\usepackage{tikz}
\usetikzlibrary{calc,arrows}  
\usepackage{enumerate,comment,enumitem}
\usepackage{subcaption}
\usepackage{utfsym}

\usepackage{babel}
\usetikzlibrary{babel}

\newdefinition{definition}{Definition}
\newdefinition{example}{Example}
\newdefinition{notation}{Notation}
\newdefinition{remark}{Remark}
\newtheorem{theorem}{Theorem}
\newtheorem{proposition}[theorem]{Proposition}
\newtheorem{lemma}[theorem]{Lemma}
\newtheorem{corollary}[theorem]{Corollary}

\makeatletter
\newcommand*\bigcdot{\mathpalette\bigcdot@{.6}}
\newcommand*\bigcdot@[2]{\mathbin{\vcenter{\hbox{\scalebox{#2}{$\m@th#1\bullet$}}}}}
\makeatother

\newcommand{\NN}{\mathbb{N}}
\newcommand{\K}{\mathbb{K}}
\newcommand{\G}{G}
\newcommand{\Tens}[1]{\mathcal{T}\langle #1\rangle}
\newcommand{\Sym}[1]{\mathcal{S}\langle #1\rangle}
\newcommand{\WGamma}[2]{\mathcal{#1}(#2)}
\newcommand{\Id}{\text{Id}}
\newcommand{\clk}{\usym{1F552}}

\tikzstyle{connexion}=[black,shorten  >= 0.075 cm ,shorten  <= 0.075 cm,line width=1,>=stealth,line width=0.5]
\tikzstyle{arete}=[red, shorten  >= 0.075 cm ,shorten  <= 0.075 cm,line width=1,>=stealth,line width=0.5,->]
\tikzstyle{GreenArete}=[green, shorten  >= 0.075 cm ,shorten  <= 0.075 cm,line width=1,>=stealth,line width=0.5,->]
\tikzstyle{BlueArete}=[blue, shorten  >= 0.075 cm ,shorten  <= 0.075 cm,line width=1,>=stealth,line width=0.5,->]
\tikzstyle{etiquette}=[red,draw, pos=0.5, fill=white, scale=0.5]
\tikzstyle{GreenEtiquette}=[green,draw, pos=0.5, fill=white, scale=0.5]
\tikzstyle{BlueEtiquette}=[blue,draw, pos=0.5, fill=white, scale=0.5]

\newcommand{\TreeCACT}{\begin{tikzpicture}[baseline,x=1cm,y=1cm]
     \filldraw[black] (0,-1) coordinate (A1) circle(0.075);
     \draw (A1) + (0,-0.2) node[scale=0.5,black]{Root};
     \filldraw[black] (0,-0.16) coordinate (A2) circle(0.075);
     \draw (A2) + (0.2,0.05) node[scale=0.5,black]{d};
     \filldraw[black] (-0.5,0.66) coordinate (A3) circle(0.075);
     \draw (A3) + (-0.2,0) node[scale=0.5,black]{b};
     \filldraw[black] (0.5,0.66) coordinate (A5) circle(0.075);
     \draw (A5) + (0,0.2) node[scale=0.5,black]{c};
     \filldraw[black] (-0.5,1.5) coordinate (A4) circle(0.075);
     \draw (A4) + (0,0.2) node[scale=0.5,black]{a};
     \draw (A1) -- (A2);
     \draw (A2) -- (A3);
     \draw (A3) -- (A4);
     \draw (A2) -- (A5);
     \end{tikzpicture}
 }

\newcommand{\WalkNonAssocNonCACT}{\begin{tikzpicture}[baseline,x=1cm,y=1cm]
     \filldraw[black] (0,-1) coordinate (A1) circle(0.075);
     \draw (A1) + (0,-0.2) node[scale=0.5,black]{$1$};
     \filldraw[black] (0,1) coordinate (A2) circle(0.075);
     \draw (A2) + (0.2,0.05) node[scale=0.5,black]{$2$};
     \filldraw[black] (-0.7,0) coordinate (A3) circle(0.075);
     \draw (A3) + (0,0.22) node[scale=0.5,black]{$3$};
     \draw[arete] (A1) to [bend right=50] node[etiquette]{$1$} (A2);
     \draw[arete] (A2) to [bend right=70] node[etiquette]{$2$} (A3);
    \draw[arete] (A3) to [in=200,out=140, loop,min distance=7mm] node[etiquette]{$3$} (A3);
     \draw[arete] (A3) to [bend right=70] node[etiquette]{$4$} (A2);
     \draw[arete] (A2) -- node[etiquette]{$5$} (A3);
      \draw[arete] (A3) to [in=220,out=280, loop,min distance=7mm] node[etiquette]{$6$} (A3);
     \draw[arete] (A3) --  node[etiquette]{$7$} (A1);
     \end{tikzpicture}
 }

\newcommand{\WalkNonAssocCACT}{\begin{tikzpicture}[baseline,x=1cm,y=1cm]
     \filldraw[black] (0,-1) coordinate (A1) circle(0.075);
     \draw (A1) + (0,-0.2) node[scale=0.5,black]{$1$};
     \filldraw[black] (0,.7) coordinate (A2) circle(0.075);
     \draw (A2) + (0.2,0.05) node[scale=0.5,black]{$2$};
     \filldraw[black] (-0.7,0) coordinate (A3) circle(0.075);
     \draw (A3) + (0,0.22) node[scale=0.5,black]{$4$};
    \filldraw[black] (0,1.5) coordinate (A4) circle(0.075);
     \draw (A4) + (0.2,0.05) node[scale=0.5,black]{$3$};
     \draw[arete] (A1) to [bend right=50] node[etiquette]{$1$} (A2);
     \draw[arete] (A2) to [bend right=50] node[etiquette]{$2$} (A4);
    \draw[arete] (A4) to [in=120,out=60, loop,min distance=7mm] node[etiquette]{$3$} (A4);
     \draw[arete] (A4) to [bend right=70] node[etiquette]{$4$} (A2);
     \draw[arete] (A2) -- node[etiquette]{$5$} (A3);
      \draw[arete] (A3) to [in=220,out=280, loop,min distance=7mm] node[etiquette]{$6$} (A3);
     \draw[arete] (A3) --  node[etiquette]{$7$} (A1);
     \end{tikzpicture}
 }

\newcommand{\WalkNonAssocCPtermD}{\begin{tikzpicture}[baseline,x=1cm,y=1cm]
     \filldraw[black] (0,-1) coordinate (A1) circle(0.075);
     \draw (A1) + (0,-0.2) node[scale=0.5,black]{$1$};
     \filldraw[black] (0,1) coordinate (A2) circle(0.075);
     \draw (A2) + (0,0.2) node[scale=0.5,black]{$2$};
     \filldraw[black] (-1,0) coordinate (A3) circle(0.075);
     \draw (A3) + (0.1,-0.15) node[scale=0.5,black]{$3$};
     \filldraw[black] (-0.8,-0.8) coordinate (A4) circle(0.075);
     \draw (A4) + (0.15,0.15) node[scale=0.5,black]{$4$};
     \draw[arete] (A1) to [bend right=50] node[etiquette]{$1$} (A2);
     \draw[arete] (A2) to[bend right=70] node[etiquette]{$5$} (A3);
     \draw[arete] (A3) to[bend right=20] node[etiquette]{$6$} (A4);
\draw[arete] (A4) to [in=200,out=140, loop,min distance=7mm] node[etiquette]{$7$} (A4);
\draw[arete] (A4) to [in=280,out=240, loop,min distance=7mm] node[etiquette]{$8$} (A4);
     \draw[arete] (A4) to[bend right=20]  node[etiquette]{$9$} (A1);
     \end{tikzpicture}
 }

 \newcommand{\WalkNonAssocCPtermd}{\begin{tikzpicture}[baseline,x=1cm,y=1cm]
     \filldraw[black] (0,0.65) coordinate (A2) circle(0.075);
     \draw (A2) + (0,0.2) node[scale=0.5,black]{$2$};
     \filldraw[black] (-1,-0.35) coordinate (A3) circle(0.075);
     \draw (A3) + (0.1,-0.15) node[scale=0.5,black]{$3$};
     \draw[arete] (A2) -- node[etiquette]{$2$} (A3);
   \draw[arete] (A3) to [in=200,out=140, loop,min distance=7mm] node[etiquette]{$3$} (A3);
     \draw[arete] (A3) to[bend right=70] node[etiquette]{$4$} (A2);
     \end{tikzpicture}
 }

 \newcommand{\WalkNonAssocCPtermC}{\begin{tikzpicture}[baseline,x=1cm,y=1cm]
     \filldraw[black] (0,-1) coordinate (A1) circle(0.075);
     \draw (A1) + (0,-0.2) node[scale=0.5,black]{$1$};
     \filldraw[black] (0,1) coordinate (A2) circle(0.075);
     \draw (A2) + (0,0.2) node[scale=0.5,black]{$2$};
     \filldraw[black] (-1,0) coordinate (A3) circle(0.075);
     \draw (A3) + (0.1,-0.15) node[scale=0.5,black]{$3$};
     \filldraw[black] (-0.8,-0.8) coordinate (A4) circle(0.075);
     \draw (A4) + (0.15,0.15) node[scale=0.5,black]{$4$};
     \draw[arete] (A1) to [bend right=50] node[etiquette]{$1$} (A2);
     \draw[arete] (A2) -- node[etiquette]{$2$} (A3);
     \draw[arete] (A3) to[bend right=70] node[etiquette]{$4$} (A2);
     \draw[arete] (A2) to[bend right=70] node[etiquette]{$5$} (A3);
     \draw[arete] (A3) to[bend right=20] node[etiquette]{$6$} (A4);
\draw[arete] (A4) to [in=200,out=140, loop,min distance=7mm] node[etiquette]{$7$} (A4);
\draw[arete] (A4) to [in=280,out=240, loop,min distance=7mm] node[etiquette]{$8$} (A4);
     \draw[arete] (A4) to[bend right=20]  node[etiquette]{$9$} (A1);
     \end{tikzpicture}
 }

 \newcommand{\WalkNonAssocCPtermc}{\begin{tikzpicture}[baseline,x=1cm,y=1cm]
     \filldraw[black] (-1,0) coordinate (A3) circle(0.075);
     \draw (A3) + (0.1,-0.15) node[scale=0.5,black]{$3$};
   \draw[arete] (A3) to [in=200,out=140, loop,min distance=7mm] node[etiquette]{$3$} (A3);
     \end{tikzpicture}
 }
 
 \newcommand{\WalkNonAssocCPtermB}{\begin{tikzpicture}[baseline,x=1cm,y=1cm]
     \filldraw[black] (0,-1) coordinate (A1) circle(0.075);
     \draw (A1) + (0,-0.2) node[scale=0.5,black]{$1$};
     \filldraw[black] (0,1) coordinate (A2) circle(0.075);
     \draw (A2) + (0,0.2) node[scale=0.5,black]{$2$};
     \filldraw[black] (-1,0) coordinate (A3) circle(0.075);
     \draw (A3) + (0.1,-0.15) node[scale=0.5,black]{$3$};
     \filldraw[black] (-0.8,-0.8) coordinate (A4) circle(0.075);
     \draw (A4) + (0.15,0.15) node[scale=0.5,black]{$4$};
     \draw[arete] (A1) to [bend right=50] node[etiquette]{$1$} (A2);
     \draw[arete] (A2) -- node[etiquette]{$2$} (A3);
    \draw[arete] (A3) to [in=200,out=140, loop,min distance=7mm] node[etiquette]{$3$} (A3);
     \draw[arete] (A3) to[bend right=70] node[etiquette]{$4$} (A2);
     \draw[arete] (A2) to[bend right=70] node[etiquette]{$5$} (A3);
     \draw[arete] (A3) to[bend right=20] node[etiquette]{$6$} (A4);
     \draw[arete] (A4) to[bend right=20]  node[etiquette]{$9$} (A1);
     \end{tikzpicture}
 }

 \newcommand{\WalkNonAssocCPtermb}{\begin{tikzpicture}[baseline,x=1cm,y=1cm]
     \filldraw[black] (0,0) coordinate (A4) circle(0.075);
     \draw (A4) + (0.15,0.15) node[scale=0.5,black]{$4$};
\draw[arete] (A4) to [in=200,out=140, loop,min distance=7mm] node[etiquette]{$7$} (A4);
\draw[arete] (A4) to [in=280,out=240, loop,min distance=7mm] node[etiquette]{$8$} (A4);
     \end{tikzpicture}
 }
 
 \newcommand{\WalkNonAssocCPtermA}{\begin{tikzpicture}[baseline,x=1cm,y=1cm]
     \filldraw[black] (0,-1) coordinate (A1) circle(0.075);
     \draw (A1) + (0,-0.2) node[scale=0.5,black]{$1$};
     \filldraw[black] (0,1) coordinate (A2) circle(0.075);
     \draw (A2) + (0,0.2) node[scale=0.5,black]{$2$};
     \filldraw[black] (-1,0) coordinate (A3) circle(0.075);
     \draw (A3) + (0.1,-0.15) node[scale=0.5,black]{$3$};
     \filldraw[black] (-0.8,-0.8) coordinate (A4) circle(0.075);
     \draw (A4) + (0.15,0.15) node[scale=0.5,black]{$4$};
     \draw[arete] (A1) to [bend right=50] node[etiquette]{$1$} (A2);
     \draw[arete] (A2) -- node[etiquette]{$2$} (A3);
    \draw[arete] (A3) to [in=200,out=140, loop,min distance=7mm] node[etiquette]{$3$} (A3);
     \draw[arete] (A3) to[bend right=70] node[etiquette]{$4$} (A2);
     \draw[arete] (A2) to[bend right=70] node[etiquette]{$5$} (A3);
     \draw[arete] (A3) to[bend right=20] node[etiquette]{$6$} (A4);
\draw[arete] (A4) to [in=200,out=140, loop,min distance=7mm] node[etiquette]{$7$} (A4);
     \draw[arete] (A4) to[bend right=20]  node[etiquette]{$9$} (A1);
     \end{tikzpicture}
 }

 \newcommand{\WalkNonAssocCPterma}{\begin{tikzpicture}[baseline,x=1cm,y=1cm]
     \filldraw[black] (0,0) coordinate (A4) circle(0.075);
     \draw (A4) + (0.15,0.15) node[scale=0.5,black]{$4$};
\draw[arete] (A4) to [in=280,out=240, loop,min distance=7mm] node[etiquette]{$8$} (A4);
     \end{tikzpicture}
 }
 
 \newcommand{\WalkNonAssocCP}{\begin{tikzpicture}[baseline,x=1cm,y=1cm]
     \filldraw[black] (0,-1) coordinate (A1) circle(0.075);
     \draw (A1) + (0,-0.2) node[scale=0.5,black]{$1$};
     \filldraw[black] (0,1) coordinate (A2) circle(0.075);
     \draw (A2) + (0,0.2) node[scale=0.5,black]{$2$};
     \filldraw[black] (-1,0) coordinate (A3) circle(0.075);
     \draw (A3) + (0.1,-0.15) node[scale=0.5,black]{$3$};
     \filldraw[black] (-0.8,-0.8) coordinate (A4) circle(0.075);
     \draw (A4) + (0.15,0.15) node[scale=0.5,black]{$4$};
     \draw[arete] (A1) to [bend right=50] node[etiquette]{$1$} (A2);
     \draw[arete] (A2) -- node[etiquette]{$2$} (A3);
    \draw[arete] (A3) to [in=200,out=140, loop,min distance=7mm] node[etiquette]{$3$} (A3);
     \draw[arete] (A3) to[bend right=70] node[etiquette]{$4$} (A2);
     \draw[arete] (A2) to[bend right=70] node[etiquette]{$5$} (A3);
     \draw[arete] (A3) to[bend right=20] node[etiquette]{$6$} (A4);
\draw[arete] (A4) to [in=200,out=140, loop,min distance=7mm] node[etiquette]{$7$} (A4);
\draw[arete] (A4) to [in=280,out=240, loop,min distance=7mm] node[etiquette]{$8$} (A4);
     \draw[arete] (A4) to[bend right=20]  node[etiquette]{$9$} (A1);
     \end{tikzpicture}
 }
 
  \newcommand{\WalkNonAssoc}{\begin{tikzpicture}[baseline,x=1cm,y=1cm]
     \filldraw[black] (0,-1) coordinate (A1) circle(0.075);
     \draw (A1) + (0,-0.2) node[scale=0.5,black]{$1$};
     \filldraw[black] (0,1) coordinate (A2) circle(0.075);
     \draw (A2) + (0,0.2) node[scale=0.5,black]{$2$};
     \filldraw[black] (-1,0) coordinate (A3) circle(0.075);
     \draw (A3) + (0,-0.2) node[scale=0.5,black]{$3$};
     \draw[arete] (A1) to [bend right=50] node[etiquette]{$1$} (A2);
     \draw[arete] (A2) -- node[etiquette]{$2$} (A3);
    \draw[arete] (A3) to [in=200,out=140, loop,min distance=7mm] node[etiquette]{$3$} (A3);
     \draw[arete] (A3) to[bend right=70] node[etiquette]{$4$} (A2);
     \draw[arete] (A2) to[bend right=70] node[etiquette]{$5$} (A3);
     \draw[arete] (A3) --  node[etiquette]{$6$} (A1);
     \end{tikzpicture}
 }

 \newcommand{\GraphKCinq}{\begin{tikzpicture}[baseline,x=1cm,y=1cm]
 	\coordinate (C) at (0,0);
 	\foreach \i in {1,...,5}{
 		\pgfmathparse{162-\i*72};
 		\pgfmathsetmacro{\ang}{\pgfmathresult};
 		\filldraw[black] (C) + (\ang:1) coordinate (A\i) circle(0.075);
 		\draw (C) + (\ang:1.2) coordinate (B\i) node[scale=0.5, black]{$\i$};
 	}
 	\foreach \i in {1,...,5}{
 		\foreach \p in {\i,...,5}{
 			\ifthenelse{\equal{\i}{\p}}{\draw[black] (B\i) circle (0.2);}{\draw[connexion] (A\i) -- (A\p);}
 		}
 	}
 	\end{tikzpicture}
 }
 
 \newcommand{\WalkOmegaSPEC}{\begin{tikzpicture}[baseline,x=1cm,y=1cm]
 	\coordinate (C) at (0,0); 
 	\foreach \i in {1,...,4}{
 		\pgfmathparse{162-\i*72};
 		\pgfmathsetmacro{\ang}{\pgfmathresult};
 		\filldraw[black] (C) + (\ang:1) coordinate (A\i) circle(0.075);
 		\draw (C) + (\ang:1.2) node[scale=0.5,black]{$\i$};
 	}
 	\draw[arete] (A1)-- node[etiquette]{$1$} (A2);
 	\draw[arete] (A2) -- node[etiquette]{$2$} (A3);
 	\draw[arete] (A3) to[bend right=60] node[etiquette]{$3$} (A2);
 	\draw[arete] (A2) to[bend right=60] node[etiquette]{$4$} (A3);
 	\draw[arete] (A3) --  node[etiquette]{$5$} (A4);
  \draw[arete] (A4)-- node[etiquette]{$6$} (A1);
 	\end{tikzpicture}
 }
 
 \newcommand{\WalkOmega}{\begin{tikzpicture}[baseline,x=1cm,y=1cm]
 	\coordinate (C) at (0,0); 
 	\foreach \i in {1,...,4}{
 		\pgfmathparse{162-\i*72};
 		\pgfmathsetmacro{\ang}{\pgfmathresult};
 		\filldraw[black] (C) + (\ang:1) coordinate (A\i) circle(0.075);
 		\draw (C) + (\ang:1.2) node[scale=0.5,black]{$\i$};
 	}
 	\draw[arete] (A1)-- node[etiquette]{$1$} (A2);
 	\draw[arete] (A2) -- node[etiquette]{$2$} (A3);
 	\draw[arete] (A3) to[bend right=60] node[etiquette]{$3$} (A2);
 	\draw[arete] (A2) to[bend right=60] node[etiquette]{$4$} (A3);
 	\draw[arete] (A3) --  node[etiquette]{$5$} (A4);
 	\end{tikzpicture}
 }
 
\newcommand{\WalkSigma}{\begin{tikzpicture}[baseline,x=1cm,y=1cm]
 	\coordinate (C) at (0,0); 
 	\foreach \i in {1,...,5}{
 		\pgfmathparse{162-\i*72};
 		\pgfmathsetmacro{\ang}{\pgfmathresult};
 		\filldraw[black] (C) + (\ang:1) coordinate (A\i) circle(0.075);
 		\draw (C) + (\ang:1.2) node[scale=0.5,black]{$\i$};
 	}
 	\draw[arete] (A1)-- node[etiquette]{$1$} (A2);
 	\draw[arete] (A2) -- node[etiquette]{$2$} (A3);
 	\draw[arete] (A3) to[bend right=60] node[etiquette]{$3$} (A2);
 	\draw[arete] (A2) --  node[etiquette]{$4$} (A4);
 	\draw[arete] (A4) --  node[etiquette]{$5$} (A5);
 	\draw[arete] (A5) --  node[etiquette]{$6$} (A2);
 	\draw[arete] (A2) to [in=110,out=50, loop,min distance=7mm] node[etiquette]{$7$} (A2);
 	\end{tikzpicture}
 }

 \newcommand{\WalkTriangle}[1]{\begin{tikzpicture}[baseline,x=1cm,y=1cm]
	\coordinate (C) at (0,0); 
	\foreach \i in {1,2,3}{
		\pgfmathparse{-180+\i*90};
		\pgfmathsetmacro{\ang}{\pgfmathresult};
		\filldraw[black] (C) + (\ang:1.2) coordinate (A\i) circle(0.075);
		\ifthenelse{\equal{#1}{}}{\draw (C) + (\ang:1.4) node[scale=0.5,black]{$\i$};}{}
	}
	\draw[arete] (A1) to[bend left=15] node[etiquette]{$1$} (A3);
	\draw[arete] (A3) to[bend left=15] node[etiquette]{$2$} (A1);
	\draw[arete] (A1) to[bend left=20] node[etiquette]{$3$} (A2);
	\draw[arete] (A2) to[bend left=20] node[etiquette]{$4$} (A3);
	\draw[arete] (A3) to[bend left=20] node[etiquette]{$5$} (A2);
	\draw[arete] (A2) to[bend left=20] node[etiquette]{$6$} (A1);
	\end{tikzpicture}
}

\newcommand{\WalkCorollaOneLoopTwoVertices}[3]{\begin{tikzpicture}[baseline,x=1cm,y=1cm]
	\foreach \j [count = \i] in #1{
		\pgfmathparse{-270+\i*180};
		\pgfmathsetmacro{\ang}{\pgfmathresult};
		\filldraw[black] (C) + (\ang:0.5) coordinate (A\i) circle(0.075);
		\ifthenelse{\equal{#3}{}}{\draw (C) + (\ang:0.7) node[scale=0.5,black]{$\j$};}{}
	}
	\foreach \j [count = \i] in #2{
		\pgfmathparse{-\i+3};
		\pgfmathsetmacro{\ind}{\pgfmathresult};
		\draw[arete] (A\i) to[bend left=30] node[etiquette]{$\j$} (A\ind);
	}
	\end{tikzpicture}
}

\newcommand{\WalkCorollaTwoLoopsThreeVertices}[3]{\begin{tikzpicture}[baseline,x=1cm,y=1cm]
	\foreach \j [count = \i] in #1{
		\pgfmathparse{30-\i*120};
		\pgfmathsetmacro{\ang}{\pgfmathresult};
		\filldraw[black] (C) + (\ang:0.7) coordinate (A\i) circle(0.075);
		\ifthenelse{\equal{#3}{}}{\draw (C) + (\ang:0.9) node[scale=0.5,black]{$\j$};}{}
	}
	\foreach \j/\k [count = \i] in #2{
		\pgfmathparse{\i+1};
		\pgfmathsetmacro{\ind}{\pgfmathresult};
		\draw[arete] (A1) to[bend left=25] node[etiquette]{$\j$} (A\ind);
		\draw[arete] (A\ind) to[bend left=25] node[etiquette]{$\k$} (A1);
	}
	\end{tikzpicture}
}

\newcommand{\WalkLadderTwoLoopsThreeVertices}[3]{\begin{tikzpicture}[baseline,x=1cm,y=1cm]
	\foreach \j [count = \i] in #1{
		\pgfmathparse{\i-2};
		\pgfmathsetmacro{\ord}{\pgfmathresult};
		\filldraw[black] (0,\ord) coordinate (A\i) circle(0.075);
		\ifthenelse{\equal{#3}{}}{\draw (A\i) + (-0.2,0) node[scale=0.5,black]{$\j$};}{}
	}
	\foreach \j/\k [count = \i] in #2{
		\pgfmathparse{\i+1};
		\pgfmathsetmacro{\ind}{\pgfmathresult};
		\draw[arete] (A\i) to[bend left=25] node[etiquette]{$\j$} (A\ind);
		\draw[arete] (A\ind) to[bend left=25] node[etiquette]{$\k$} (A\i);
	}
	\end{tikzpicture}
}

\newcommand{\MakeCircle}[4]{
	\foreach \j [count = \i] in #2{
		\pgfmathparse{#4-(\i-1)*360/#1};
		\pgfmathsetmacro{\ang}{\pgfmathresult};
		\filldraw[black] (C) + (\ang:1) coordinate (A\i) circle(0.075);
		\draw (C) + (\ang:1.2) node[scale=0.5, black]{$\j$};
	}
	\foreach \k [count = \i] in #3{
		\pgfmathparse{#4-(\i-1)*360/#1};
		\pgfmathsetmacro{\angOne}{\pgfmathresult};
		\pgfmathparse{\angOne-360/#1};
		\pgfmathsetmacro{\angTwo}{\pgfmathresult};
		\draw[arete] (A\i) arc (\angOne:\angTwo:1) node[etiquette]{$\k$};
	}
}

\newcommand{\WalkCircleSelfAvoiding}[3]{\begin{tikzpicture}[baseline,x=1cm,y=1cm]
	\MakeCircle{#1}{#2}{#3}{-90}
\end{tikzpicture}
}

\newcommand{\WalkSelfAvoiding}[3]{\begin{tikzpicture}[baseline,x=1cm,y=1cm]
	\foreach \k [count= \i] in #1{
		\filldraw[black] (\i,0) coordinate (A\i) circle(0.075);
		\draw (A\i) + (0,-0.2) node[scale=0.5, black]{$\k$};
	}
	\foreach \k [count=\i] in #2{
		\pgfmathparse{\i+1};
		\pgfmathsetmacro{\suiv}{\pgfmathresult};
		\draw[arete] (A\i) -- (A\suiv) node[etiquette]{$\k$};
	}
	
	\foreach \k in #3{
		\pgfmathparse{\k+1};
		\pgfmathsetmacro{\suiv}{\pgfmathresult};
		\draw[GreenArete] (A\k) -- (A\suiv) node[GreenEtiquette]{$\k$};
	}
	\end{tikzpicture}
}

 \newcommand{\WalkGirl}{\begin{tikzpicture}[baseline,x=1cm,y=1cm]
	\coordinate (C) at (0,0); 
	\foreach \i in {2,3,4}{
		\pgfmathparse{360-\i*90};
		\pgfmathsetmacro{\ang}{\pgfmathresult};
		\filldraw[black] (C) + (\ang:1) coordinate (A\i) circle(0.075);
		\draw (C) + (\ang:1.2) node[scale=0.5,black]{$\i$};
	}
	\filldraw[black] (-1.5,-0.8) coordinate (A1)  circle(0.075) (A1)+(-0.2,-0.05) node[below,scale=0.5,black]{$1$};
	\filldraw[black] (1.5,-0.8) coordinate (A5)  circle(0.075) (A5)+(0.2,-0.05) node[below,scale=0.5,black]{$5$};
	\draw[arete] (A1)-- node[etiquette]{$1$} (A2);
	\draw[arete] (A2) to[bend left=40] node[etiquette]{$2$} (A3);
	\draw[arete] (A3) to[bend left=40] node[etiquette]{$3$} (A2);
	\draw[arete] (A2) arc (-180:0:1) node[etiquette]{$4$} (A4);
	\draw[arete] (A4) to[bend left=40] node[etiquette]{$5$} (A3);
	\draw[arete] (A3) to[bend left=40] node[etiquette]{$6$} (A4);
	\draw[arete] (A4) --  node[etiquette]{$7$} (A5);
	\end{tikzpicture}
}

\newcommand{\ClosedWalkMickey}[5]{\begin{tikzpicture}[baseline,x=1cm,y=1cm]
	\MakeCircle{3}{#2}{#3}{-90}
	\ifthenelse{\equal{#1}{}}{\draw[arete] (A2) to [in=100,out=190, loop,min distance=15mm] node[etiquette]{$#4$} (A2);
		\draw[arete] (A3) to [in=0,out=100, loop,min distance=15mm] node[etiquette]{$#5$} (A3);}{}
	\ifthenelse{\equal{#1}{1}}{\draw[arete] (A3) to [in=0,out=100, loop,min distance=15mm] node[etiquette]{$#5$} (A3);}{}
	\ifthenelse{\equal{#1}{2}}{\draw[arete] (A2) to [in=100,out=190, loop,min distance=15mm] node[etiquette]{$#4$} (A2);}{}
	\end{tikzpicture}
}

\newcommand{\OpenWalkMickey}{\begin{tikzpicture}[baseline,x=1cm,y=1cm]
	\MakeCircle{4}{{0,l,k,\ell}}{{\omega^{0k},\omega^{k'l},\omega^{l'\ell}}}{-135}
	\draw[arete] (A2) to [in=100,out=190, loop,min distance=15mm] node[etiquette]{$\omega^{kk'}$} (A2);
	\draw[arete] (A3) to [in=0,out=100, loop,min distance=15mm] node[etiquette]{$\omega^{ll'}$} (A3);
	\end{tikzpicture}
}

\newcommand{\MakeLadder}[4]{
	\foreach \j [count=\i] in #2{
		\pgfmathparse{\i-1};
		\pgfmathsetmacro{\prev}{\pgfmathresult};
		\filldraw[black] (0,\prev) coordinate (A\i) circle(0.075);
		\ifthenelse{\i>1}{\draw (A\i)++(-0.3,0) node[scale=0.7,black]{$\j$};}{\draw (A\i)++(0,-0.3) node[scale=0.7,black]{$\j$};}
		
	}
	\foreach \k/\p [count = \i] in #3{
		\pgfmathparse{\i+1};
		\pgfmathsetmacro{\suiv}{\pgfmathresult};
		\draw[arete] (A\i) to[bend left=60] node[sloped, etiquette]{$\k$} (A\suiv);
		\draw[arete] (A\suiv) to[bend left=60] node[sloped, etiquette]{$\p$} (A\i);
	}
	\draw[arete] (A#1) to [in=45,out=135, loop,min distance=15mm] node[etiquette]{$#4$} (A#1);
}

\newcommand{\GeneralClosedWalkLadder}[3]{\begin{tikzpicture}[baseline,x=1cm,y=1cm]
	\foreach \j [count=\i] in #1{
		\pgfmathparse{\i-1};
		\pgfmathsetmacro{\prev}{\pgfmathresult};
		\filldraw[black] (0,\prev) coordinate (A\i) circle(0.075);
		\ifthenelse{\i>1}{\draw (A\i)++(-0.3,0) node[scale=0.7,black]{$\j$};}{\draw (A\i)++(0,-0.3) node[scale=0.7,black]{$\j$};}	
	}
	\foreach \k/\p [count=\i] in #2{
		\ifthenelse{\equal{\i}{1}}{\draw[arete] (A1) to[bend left=60] node[sloped, etiquette]{$\k$} (A2);
			\draw[arete] (A2) to[bend left=60] node[sloped, etiquette]{$\p$} (A1);}{\draw[arete] (A3) to[bend left=60] node[sloped, etiquette]{$\k$} (A4);	\draw[arete] (A4) to[bend left=60] node[sloped, etiquette]{$\p$} (A3);}
	}
	\draw[arete] (A4) to [in=45,out=135, loop,min distance=15mm] node[etiquette]{$#3$} (A4);
	\draw[arete] ($(A2)!0.5!(A3)$) node{$\vdots$} (A3);
	\end{tikzpicture}
}

\newcommand{\OpenWalkLadder}{\begin{tikzpicture}[baseline,x=1cm,y=1cm]
	\MakeLadder{2}{{k,l}}{{\omega^{kl}/\omega{l'k'}}}{\omega^{ll'}}
	\filldraw[black] (-1,0) coordinate (A3) circle(0.075); 
	\draw (A3)++(-0.2,0) node[scale=0.5,black]{$0$};
	\filldraw[black] (1,0) coordinate (A4) circle(0.075);
	\draw (A4)++(0.2,0) node[scale=0.5,black]{$\ell$};
	\draw[arete] (A3) --  node[etiquette]{$\omega^{0k}$} (A1);
	\draw[arete] (A1) --  node[etiquette]{$\omega^{k'\ell}$} (A4);
	\end{tikzpicture}
}

\newcommand{\MakeCorolla}[3]{
	\filldraw[black] (0,0) coordinate (C) circle(0.075);
	\draw (C)++(0,-0.3) node[scale=0.7,black]{$#2$};
	\pgfmathparse{180/#1};
	\pgfmathsetmacro{\ang}{\pgfmathresult};
	\foreach \j [count=\i] in #3{
		\ifthenelse{\equal{#1}{1}}{\draw[arete] (C) to [in=45,out=135, loop,min distance=15mm] node[etiquette]{$\j$} (C);}{}
		\ifthenelse{\equal{#1}{2}}{\ifthenelse{\equal{\i}{1}}{\draw[arete] (C) to [in=105,out=150, loop,min distance=15mm] node[etiquette]{$\j$} (C);}{\draw[arete] (C) to [in=30,out=75, loop,min distance=15mm] node[etiquette]{$\j$} (C);}}{}
		\ifthenelse{#1>2}{\pgfmathparse{180-(\i-1)*\ang};
			\pgfmathsetmacro{\angOne}{\pgfmathresult};
			\pgfmathparse{\angOne-\ang};
			\pgfmathsetmacro{\angTwo}{\pgfmathresult};
			\draw[arete] (C) to [in=\angTwo,out=\angOne, loop,min distance=15mm] node[etiquette]{$\j$} (C);}{}
	}
}

\newcommand{\ClosedWalkCorolla}[3]{\begin{tikzpicture}[baseline,x=1cm,y=1cm]
	\MakeCorolla{#1}{#2}{#3}
	\end{tikzpicture}
}

\newcommand{\GeneralClosedWalkCorolla}[3]{\begin{tikzpicture}[baseline,x=1cm,y=1cm]
	\filldraw[black] (0,0) coordinate (C) circle(0.075);
	\draw (C)++(0,-0.3) node[scale=0.7,black]{$#1$};
	\draw[arete] (C) to [in=115,out=160, loop,min distance=15mm] node[etiquette]{$#2$} (C);
	\draw[arete] (C) to [in=20,out=65, loop,min distance=15mm] node[etiquette]{$#3$} (C);
	\draw[red] (0,0.5) node{$\cdots$};
	\end{tikzpicture}
}

\newcommand{\OpenWalkCorolla}{\begin{tikzpicture}[baseline,x=1cm,y=1cm]
	\MakeCorolla{2}{k}{{\omega^{kl},\omega{ll'}}}
	\filldraw[black] (-1.3,0) coordinate (A1) circle(0.075); 
	\draw (A1)++(-0.2,0) node[scale=0.5,black]{$0$};
	\filldraw[black] (1.3,0) coordinate (A2) circle(0.075);
	\draw (A2)++(0.2,0) node[scale=0.5,black]{$\ell$};
	\draw[arete] (A1) --  node[etiquette]{$\omega^{0k}$} (C);
	\draw[arete] (C) --  node[etiquette]{$\omega^{l'\ell}$} (A2);
	\end{tikzpicture}
}

\newcommand{\WalkExampleOrder}[2]{\begin{tikzpicture}[baseline,x=1cm,y=1cm]
	\filldraw[black] (-2,0) coordinate (A1) circle(0.075);
	\filldraw[black] (-1,0) coordinate (A2) circle(0.075);
	\filldraw[black] (-1.75,0.75) coordinate (A3) circle(0.075);
	\filldraw[black] (0,0) coordinate (A4) circle(0.075);
	\filldraw[black] (1,0) coordinate (A5) circle(0.075);
	\filldraw[black] (1,1) coordinate (A6) circle(0.075);
	\filldraw[black] (2,0) coordinate (A7) circle(0.075);
	\ifthenelse{\equal{#2}{}}{}{
		\foreach \i [count=\k] in#2{
			\ifthenelse{\equal{\k}{3}}{	\draw (A\k) + (-0.15,-0.15) node[scale=0.5,black]{$\i$};}{
				\ifthenelse{\equal{\k}{6}}{\draw (A\k) + (-0.2,0) node[scale=0.5,black]{$\i$};}{\draw (A\k) + (0,-0.2) node[scale=0.5,black]{$\i$};}
			}
		}
	}
	\draw[arete] (A1) -- node[etiquette]{$1$} (A2);	
	\draw[arete] (A2) to[bend left=30] node[etiquette]{$2$} (A3);
	\ifthenelse{\equal{#1}{}}{\draw[arete] (A3) to [in=105,out=150, loop,min distance=10mm] node[etiquette]{$3$} (A3);
		\draw[arete] (A3) to [in=30,out=75, loop,min distance=10mm] node[etiquette]{$4$} (A3);
	}{}
	\draw[arete] (A3) to[bend left=30] node[etiquette]{$5$} (A2);
	\draw[arete] (A2) to [in=65,out=110, loop,min distance=10mm] node[etiquette]{$6$} (A2);
	\draw[arete] (A2) to [in=20,out=65, loop,min distance=10mm] node[etiquette]{$7$} (A2);
	\draw[arete] (A2) -- node[etiquette]{$8$} (A4);		
	\draw[arete] (A4) -- node[etiquette]{$9$} (A5);
	\draw[arete] (A5) to[bend left=30] node[etiquette]{$10$} (A6);
	\ifthenelse{\equal{#1}{}}{\draw[arete] (A6) to [in=45,out=135, loop,min distance=10mm] node[etiquette]{$11$} (A6);}{}
	\draw[arete] (A6) to[bend left=30] node[etiquette]{$12$} (A5);
	\draw[arete] (A5) -- node[etiquette]{$13$} (A7);
	\end{tikzpicture}
}

\newcommand{\WalkExampleAntipode}[1]{\begin{tikzpicture}[baseline,x=1cm,y=1cm]
	\filldraw[black] (-2,0) coordinate (A1) circle(0.075);
	\draw (A1) + (0,-0.2) node[scale=0.5,black]{$1$};
	\filldraw[black] (-1,0) coordinate (A2) circle(0.075);
	\draw (A2) + (0,-0.2) node[scale=0.5,black]{$2$};
	\filldraw[black] (0,0) coordinate (A3) circle(0.075);
	\draw (A3) + (0,-0.2) node[scale=0.5,black]{$3$};
	\filldraw[black] (1,0) coordinate (A4) circle(0.075);
	\draw (A4) + (0,-0.2) node[scale=0.5,black]{$4$};
	\filldraw[black] (2,0) coordinate (A5) circle(0.075);
	\draw (A5) + (0,-0.2) node[scale=0.5,black]{$5$};
	\draw[arete] (A1) -- node[etiquette]{$1$} (A2);	
	\draw[arete] (A2) -- node[etiquette]{$4$} (A3);		
	\draw[arete] (A3) -- node[etiquette]{$5$} (A4);		
	\draw[arete] (A4) -- node[etiquette]{$7$} (A5);
	\ifthenelse{\equal{#1}{1}}{
		\draw[arete] (A2) to [in=105,out=150, loop,min distance=10mm] node[etiquette]{$2$} (A2);
		\draw[arete] (A2) to [in=30,out=75, loop,min distance=10mm] node[etiquette]{$3$} (A2);
		\draw[arete] (A4) to [in=45,out=135, loop,min distance=10mm] node[etiquette]{$6$} (A4);}{}	
	\ifthenelse{\equal{#1}{2}}{
		\draw[arete] (A2) to [in=45,out=135, loop,min distance=10mm] node[etiquette]{$2$} (A2);
		\draw[arete] (A4) to [in=45,out=135, loop,min distance=10mm] node[etiquette]{$6$} (A4);}{}
	\ifthenelse{\equal{#1}{3}}{
		\draw[arete] (A4) to [in=45,out=135, loop,min distance=10mm] node[etiquette]{$6$} (A4);}{}
	\ifthenelse{\equal{#1}{4}}{
		\draw[arete] (A2) to [in=105,out=150, loop,min distance=10mm] node[etiquette]{$2$} (A2);
		\draw[arete] (A2) to [in=30,out=75, loop,min distance=10mm] node[etiquette]{$3$} (A2);}{}
	\ifthenelse{\equal{#1}{5}}{
		\draw[arete] (A2) to [in=45,out=135, loop,min distance=10mm] node[etiquette]{$2$} (A2);}{}
	\end{tikzpicture}
}

\begin{document}
\allowdisplaybreaks
\begin{frontmatter}
\author[1]{Loïc Foissy}
\ead{foissy@univ-littoral.fr}

\author[1]{Pierre-Louis Giscard\corref{cor1}}
\ead{giscard@univ-littoral.fr}
\affiliation[1]{organization={Université du Littoral Côte d'Opale, UR 2597, LMPA, Laboratoire de Mathématiques Pures et Appliquées Joseph Liouville},
addressline={50 rue F. Buisson},
postcode={F-62100},
city={Calais},
country={France}}

\author[2]{Cécile Mammez}
\ead{cecile.mammez@ac-versailles.fr}
\affiliation[2]{organization={Laboratoire de Mathématiques de Reims - UMR 9008, U.F.R. Sciences Exactes et Naturelles, Université de Reims},
addressline={Moulin de la Housse - BP 1039},
city={Reims},
postcode={51687 cedex 2},
country={France}}

\cortext[cor1]{Corresponding author.}
	

		\title{A co-preLie structure from chronological loop erasure in graph walks}
	
	
	\begin{abstract} 
 We show that the chronological removal of cycles from a walk on a graph, known as Lawler's loop-erasing procedure, generates a preLie co-algebra on the vector space spanned by the walks. In addition, we prove that the tensor and symmetric algebras of graph walks are graded Hopf algebras, provide their antipodes explicitly and recover the preLie co-algebra from a brace coalgebra on the tensor algebra of graph walks. Finally we exhibit sub-Hopf algebras associated to particular types of walks. 
	\end{abstract}
 	\end{frontmatter}

\section*{Introduction} \label{section:introduction}  

Graphs and walks are ubiquitous objects in combinatorics, discrete mathematics and beyond: they appear throughout linear algebra, differential calculus and have found widespread applications in physics, engineering and biology. Yet, while graph theory is being developed, less attention has been devoted to the walks themselves, a walk being a contiguous succession of directed edges on a graph. In particular the algebraic structures associated to walks have not, to the best of our knowledge, been fully explored. We may here refer the reader to quivers and path algebras and hike monoids \cite{Giscard2017}. The goal of the present work is to exhibit a co-preLie structure naturally associated to walks on graphs (simple graphs, multi-graphs, digraphs and hypergraphs). The structure arises from a simple procedure, now known as \textit{Lawler's loop erasing} \cite{Lawler1999}, first conceived in the context of percolation theory to randomly generate simple paths--walks where all vertices are distinct--from a sample of random walks. The procedure consists of a chronological removal of cycles (called loops in Lawler's original work) as one walks along on the graph: consider for instance the complete graph $K_4$ on 4 vertices and label these vertices with integers 1 through 4. Walking along the path $1\to 2\to 1 \to 3\to 4 \to 3\to 1\to 3$ on the graph and removing cycles whenever they appear, we are left with the simple path $1\to 3$  after having successively `erased' the cycles $1\to 2\to 1$, then  $3\to 4 \to 3$ and finally  $1\to 3 \to 1$. Note how $1\to 3\to 1$ does not appear contiguously in the original walk. 
Once terminated, Lawler's loop-erasing has eliminated a set of cycles, all of whose internal vertices are distinct, leaving a possibly trivial walk-skeleton behind. If the initial walk was itself a cycle, this skeleton is the empty walk on the initial vertex (also called length-0 walk) and otherwise it is a self-avoiding path. Remark that because the loop-removal occurs in a chronological fashion, Lawler's process is strongly non-Markovian:  complete knowledge of all the past steps of a walk is required to decide the current and future erased sections at any point of the walk. 

We show below that this intuitive process is naturally associated with a co-preLie coproduct. In addition, slightly relaxing the chronological constraints by allowing simultaneous erasures under some compatibility conditions leads to Hopf algebra structures on the tensor and symmetric algebras of graph walks.\\[-.5em]

The article is organized as follows. In Section~\S\ref{DefinitionsSection} we begin with basic notations and definitions concerning walks, graphs and Lawler's loop erasing procedure. In \S\ref{Chrono} we describe the chronological structure that walks acquire from Lawler's process and use this structure to define the admissible cuts of a walk. We show in particular that this notion is well defined in the sense that in spite of the
strong chronological constraints created by Lawler’s process, cutting out admissible cuts does not alter
the other cuts admissibility. This leads in \S\ref{sectionCoproduct} to the definition of a co-product on walks which we show to be co-preLie. Then, in \S\ref{HopfTensor}, considering a wider set of simultaneously admissible cuts, called extended admissible cuts, we construct a co-associative co-product on the tensor and symmetric algebras generated by the vector space of walks on a graph. We then prove an explicit formula for the antipode maps in the so-obtained Hopf algebras. In \S\ref{BraceSection} we construct a brace coalgebra and a codendriform bialgebra on the tensor algebra generated by graph walks and use these to recover the preLie structure as a corollary of the Hopf algebra of the preceding section.
Finally in \S\ref{CactiTowerCorollas} we exhibit  Hopf subalgebras associated to certain types of walks, the cacti, towers and corollas.\\[-.5em]

In a subsequent work inspired by previous combinatorial results \cite{Giscard2012}, we will show that Lawler's process is also naturally associated with a non-associative permutative product, known as nesting \cite{Giscard2012}, which satisfies the Livernet compatibility condition \cite{Livernet2006} with the co-preLie co-product defined here. This will provide the very first concrete example of the NAP - co-preLie operad in a `living' context. This construction appears to be of paramount importance given the pervasive use of graph-walks in mathematics and mathematical-physics. In particular, we will show that this leads to a useful bridge between formal sums over infinite families of walks and branched continued fractions.

	\section{Notations and definitions}\label{DefinitionsSection}
 \subsection{Notations for graphs and rooted walks}
 While we begin by recalling standard definitions for graphs, we introduce somewhat less common concepts for walks, of which we advise the reader to take special notice.\\[-.5em]

 A \emph{graph} $G=(V,E)$ is a countable set of vertices $V$ and a countable set $E$ of distinct paired vertices, called edges, denoted $\{i,j\}$, $i,j\in V$. 
  A \emph{digraph} $G=(V,E)$ is a finite set of vertices $V$ and a finite set $E\subseteq V^2$ of  {\em directed edges} (or {\em arcs}), denoted $(i,j)$ for the arc from $i$ to $j$. 
 A \emph{directed multigraph} (or \emph{multidigraph}) is defined the same way as a digraph, except that $E$ is a multiset. An edge of~$E$ is then denoted $(i,j)_k$, the integer $k$ specifying which edge from $i$ to $j$ we consider. In the present work we always assume that $G$ is non-empty.\\[-.5em] 

A \textit{rooted walk}, or rooted path, of length $\ell$ from vertex $i$ to vertex $j$ on a multi directed graph $G$ is a contiguous sequence of $\ell$ arcs starting from $i$ and ending in $j$, e.g. $\omega = (i, i_1)_{k_1}(i_1, i_2)_{k_2} \cdots (i_{\ell-1}, j)_{k_\ell}$ (a sequence of arcs is said to be contiguous if each arc but the first one starts where the previous ended). 
The rooted walk $\omega$ is \textit{open} if $i \neq j$ and \textit{closed} otherwise, in which case it is also called {\em rooted cycle}. Since we only consider rooted walks in this work, we shorten this terminology to \textit{walks}. On digraphs we may unambiguously represent walks simply as ordered sequences of vertices $\omega=w_0w_1\cdots w_{\ell-1}w_\ell$. The walk $\omega=w_0$ of length 0 is called the \emph{trivial} walk on vertex $w_0$, it is both open and closed. The set of all walks on a graph $G$ is denoted $\mathcal{W}(G)$. \\[-.5em] 

Consider a walk $\omega=w_0\dots w_\ell$. A \textit{subwalk} of a walk $\omega=w_0\cdots w_\ell$ is any walk $w_k\cdots w_{k'}$ where $0\leq k\leq k'\leq \ell$. If $k\neq k'$ and $w_k=w_{k'}$, we designate by $\omega^{k,k'}:= w_k w_{k+1}\dots w_{k'}$ the \textit{closed subwalk} of $\omega$ with root $w_k$. In a complementary way, we define the \emph{remainder} section $\omega_{k,k'} := w_0\dots w_k w_{k' +1}\dots w_\ell$ to be what remains of $\omega$ after removal of the section $\omega^{k,k'}$. Note, for convenience we denote $\omega_{k,k'}^{l,l'}$ for $(\omega_{k,k'})^{l,l'}$, the section $w_l\dots w_{l'}$ erased from the remainder $\omega_{k,k'}=w_0\dots w_{k}w_{k'+1}\dots w_\ell$. This means in particular that in $\omega_{k,k'}^{l,l'}$, integers $k,k',l$ and $l'$ all refer to indices from $\omega$.\\[-.5em]

A rooted walk in which all vertices are distinct is said to be a {\em simple path} or {\em self-avoiding walk}. The set of all such walks on a digraph $G$ is denoted $\text{SAW}(G)$.
Similarly, a rooted cycle $(i_0, i_1)_{k_1}(i_1, i_2)_{k_2} \cdots (i_{\ell-1}, i_0)_{k_\ell}$ of non-zero length for which all vertices $i_t$ are distinct is said to be a {\em simple cycle} or {\em self-avoiding polygon}. Note that a self-loop $(i, i)_k$ is considered a rooted simple cycle of length one. The set of all simple cycles on $G$ is $\text{SAP}(G)$. 
For $G$ any (directed multi)graph, to ease the notation, we also denote by $\mathcal{W}(G)$ the $\K$-vector space spanned by all walks on $G$, $\K$ being a field of characteristic $0$. For a walk $\omega \in {\mathcal W}(G)$, we designate by $V(\omega)$ the support of $\omega$, that is the set of \textit{distinct} vertices visited by $\omega$; and by $E(\omega)$ the \textit{multiset} of directed edges visited by $\omega$.\\[-.5em]

  \subsection{Definitions for loop-erasure}
	As stated in the introduction, Lawler's loop-erasing procedure consists in erasing all cycles from a walk $\omega$ in the \emph{chronological} order in which they appear. Formally, it is a selection-quotient process which transforms a walk into its self-avoiding skeleton. To construct the algebraic structures associated with Lawler's procedure we must not only consider its end product but also what it produces during its intermediary stages and what it removes from the walk, in its original context:

\begin{definition}[Loop-erased sections]\label{DefLES}
Let $G$ be a digraph and consider $\omega = w_0\dots w_\ell\in \mathcal{W}(G)$.  The set $\text{LES}(\omega)$ of loop-erased sections is the set of all \textit{closed subwalks} of $\omega$ \textit{erased} by Lawler's procedure. 
\end{definition}

	\begin{example}\label{Example1}
		On the complete graph $K_5$ on 5 vertices (including self-loops), consider the walk
		$\omega=12324522,$
		$$\omega=\WalkSigma$$
  In this illustration, the integers in boxes in the middle of edges give these edges' order of traversal while vertices are labeled by black integers next to them. 
The simple cycles erased by Lawler's procedure are $\omega^{1,3}=232$, $\omega^{3,6}=2452$ and $\omega^{6,7}=22$ and the set of erased closed subwalks of $\omega$ is therefore,
\begin{align*}
				\text{LES}(12324522)=&\{\omega^{1,3},\omega^{3,6}, \omega^{1,6}, \omega^{6,7}, \omega^{3,7}, \omega^{1,7} \}\\
				=&\{232, 2452, 232452, 22, 24522, 2324522\}.
			\end{align*}
 \end{example}

  
 
  \begin{remark}
      The requirement that the closed subwalks of $\text{LES}(\omega)$ be constructed solely from \emph{erased} sections is crucial. For example, in $$\omega=1232341=\WalkOmegaSPEC$$ we have $\text{LES}(\omega)=\{\omega^{1,3},\omega\}=\{232,1232341\}$. In particular $323\notin\text{LES}(\omega)$ because $323$ was not erased at once by Lawler's procedure. Including it would violate the chronological condition innate in Lawler's process, as erasing $323$ from $\omega$ would imply having overlooked the cycle $232$ which was closed \emph{prior} to $323$. We formalize this observation with the notion of loop-erased walks: 
  \end{remark}

\begin{definition}[Loop-erased walks]
Let $G$ be a digraph and $\omega=w_0\cdots w_\ell\in\mathcal{W}(G)$ of length $\ell$. For $0\leq k\leq \ell$, we designate  $\text{LEW}_k(\omega)$, called loop-erased walk $\omega$ at step $k$, to be what is left of $\omega$ after its first $k$ steps while performing Lawler's procedure. 
\end{definition}

	By the definitions of $\text{LES}(\omega)$ and $\text{LEW}(\omega)$ we obtain what was remarked above, namely that loop-erased sections may not straddle over one-another, a consequence of their step-by-step erasure in chronological order: 

	\begin{lemma}\label{PropLEWLES}
		Let $G$ be a digraph and $\omega = w_0\dots w_\ell\in\mathcal{W}(G)$. Then $\omega^{k,k'}\in \emph{\text{LES}}(\omega)$ if and only if there does not exist a pair of integers $0\leq l <k<l' <k'\leq \ell$ with $w_k=w_{k'}\neq w_{l}=w_{l'}$ and $\omega^{l,l'}\in \emph{\text{LES}}(\omega) $
	\end{lemma}
Before we prove the lemma, we remark that the notion of loop-erased walks allows for an alternative but equivalent definition of that of loop-erased section: 
\begin{remark}[A recursive procedure for constructing LES($\omega$)]\label{AltLES}
		Let $G$ be a digraph and consider $\omega = w_0\dots w_\ell\in \mathcal{W}(G)$.  The set $\text{LES}(\omega)$ of loop-erased sections of $\omega=w_0\cdots w_\ell$ is constructed recursively as follows. Initialize with $\text{LES}(\omega)=\emptyset$. Then for $k\in\{1,\dots, \ell-1\}$, if $w_{k+1}\in V(\text{LEW}_k(\omega))$ denote $k'$, the greatest integer such that $0\leq k'\leq k$ and $w_{k'}=w_{k+1}$. If $k'$ exists, then:  
		\begin{enumerate}
			\item add the closed walk $\omega^{k',k+1}=w_{k'}\dots w_{k+1}$ to $\text{LES}(\omega)$;
			\item if there exists $\omega^{k'',k'}\in\text{LES}(\omega)$, add the closed walk $\omega^{k'',k+)}=w_{k''}\dots w_{k+1}$ to $\text{LES}(\omega)$ as well.
		\end{enumerate}
  While equivalent to Definition~\ref{DefLES}, the above formulation is more formal in flavor and recursive in nature, thus better suited to algorithm designs and easier to wield in proofs.
\end{remark}


\begin{proof}[Proof of Lemma~\ref{PropLEWLES}]
Assuming that $\omega^{k,k'},\omega^{l,l'}\in\textrm{LES}(\omega)$, suppose that both sections nonetheless straddle over one-another. We may choose wlog that $k<l<k'<l'$. In particular, there is no earlier step $m<k$ with $w_m=w_l$ since otherwise we would effectively be in the straddling situation where $l<k$.
Then at step $l'-1\geq k'$ of the walk, vertex $w_{l}\notin V(\text{LEW}_{l'-1})(w)$ since at this point $\omega^{k,k'}$ has already been erased and so by Remark~\ref{AltLES}, $\omega^{l,l'}\notin \textrm{LES}(\omega)$, a contradiction. 
\end{proof}

\section{The chronological structure of walks}\label{Chrono}

    From a walk $\omega$, Lawler's process, once terminated, produces a set of erased simple cycles and one self-avoiding skeleton (possibly trivial). It is therefore natural to seek a co-product which to the walk $\omega$ would associate a sum over erased sections $\omega^{k,k'}$ and associated remainders $\omega_{k,k'}$, so that $\omega$ could be obtained back from these through grafting of the former onto the latter. 
    The `grafting' product appropriate to that end, known as \textit{nesting}, was first identified thanks to purely combinatorial considerations \cite{Giscard2012} and is permutative non-associative reflecting Lawler's process' chronological constraints. It is difficult to maintain any form of compatibility with nesting via such an indiscriminate procedure as cutting out all loop-erased sections however, as 
    not all pairs $(\omega^{k,k'},\,\omega_{k,k'})$ can be consistently grafted back to form the original walk; and when grafting is possible, it may be so in more than one way. These problems arise from certain towers and all corollas, respectively. Consider first an instance of the former, 
   \[\omega=1233231=\WalkNonAssoc,\]
   which is a tower in the sense that the self-loop $33$ is attached `on top of' cycle $232$, itself attached to the `base' triangle $1231$. Here $\omega^{2,3}=33$ is a valid loop-erased section of $\omega$, yet can be grafted back onto $\omega_{2,3}$ in two distinct ways: one producing $\omega$ and the other yielding the walk $\omega'=1232331$. Remark how in $\omega'$, the self-loop $33$ occurs one level below its original location in $\omega$ since it is now attached directly to the `base' triangle $1231$. Algebraically such instances correspond to cases where the nesting product fails to be associative.
    Second, for the issue with corollas, i.e. bouquets of closed walks with the same root, consider e.g.   
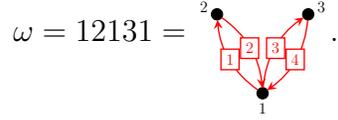
     \[\omega=12131=\WalkCorollaTwoLoopsThreeVertices{{1,2,3}}{{1/2,3/4}}{}.\]
    Here both $121,\,131\in\text{LES}(\omega)$; yet cutting e.g. $\omega^{0,2}=121$ and grafting it back onto $\omega_{0,2}=131$ either gives back the walk $\omega=12131$ or the completely different one $\omega'=13121$.
    Algebraically, these instances translate into cases where the nesting product fails to be commutative.

  
    	\subsection{Admissible cuts}
	To resolve the difficulties mentioned above, which become extensive when taken together in arbitrary long walks, we must refine the set of loop-erased sections that can be cut out of the original walk by the co-product. Here, as earlier, the major hurdle is due to the chronological constraints inherent to Lawler's process. Because of this, special attention must be paid to erased sections that appear within longer erased sections, the latter providing the temporal context of the former: 
 
 	\begin{definition}[Temporal context of an erased section]
		Let $G$ be a digraph, $\omega\in\mathcal{W}(G)$ and $\omega^{k,k'}\in\text{LES}(\omega)$. We denote $\text{LES}(\omega)^{<}_{k,k'}\subset \text{LES}(\omega)$ the subset of loop-erased sections $\omega^{l,l'}$ which \textit{strictly} include $\omega^{k,k'}$ as subwalk, i.e. $l\leq k<k'<l'$.
		Because we require $k'<l'$ strictly, $\text{LES}(\omega)^{<}_{k,k'}$ may be empty. Otherwise, we denote $\omega^{\text{min}}_{k,k'}$ the smallest element of $\text{LES}(\omega)^{<}_{k,k'}$ for inclusion.
	\end{definition}
 
By construction, if $\omega^{\text{min}}_{k,k'}$ exists, it is the tightest erased section which  comprises $\omega^{k,k'}$ entirely. It provides the relevant temporal context for $\omega^{k,k'}$ since anything outside of $\omega^{\text{min}}_{k,k'}$ creates no further chronological constraints on $\omega^{k,k'}$ beyond those on  $\omega^{\text{min}}_{k,k'}$. This is because vertices appearing in the loop-erased walk at the start of $\omega^{\text{min}}_{k,k'}$ cannot appear again inside of it by Lemma~\ref{PropLEWLES}, so are necessarily avoided by $\omega^{k,k'}$. Hence, any additional constraint that Lawler's process imposes on $\omega^{k,k'}$ as compared to $\omega^{\text{min}}_{k,k'}$ arise solely from within $\omega^{\text{min}}_{k,k'}$.
 
	\begin{example}
Let $\omega=12324522$ be the walk of Example~\ref{Example1} and consider its loop-erased section $\omega^{1,3}$. Since \(\text{LES}(12324522)=\{\omega^{1,3},\,\omega^{3,6},\, \omega^{1,6},\, \omega^{6,7},\, \omega^{3,7},\, \omega^{1,7} \} \), then $\text{LES}(\sigma)^{<}_{1,3}=\{\sigma^{1,6},\, \sigma^{1,7}\}$. Indeed, both section $\omega^{1,6}$ and $\omega^{1,7}$ strictly contain $\omega^{2,4}$. Furthermore, the smallest of these by inclusion is $\omega^{\text{min}}_{1,3}=\omega^{1,6}$, i.e. $\omega^{1,6}$ is the shortest loop-erased section strictly containing $\omega^{1,3}$.
At the opposite, there is no loop-erased section strictly containing $\omega^{3,7}\in\text{LES}(\omega)$, that is $\text{LES}(\omega)^{<}_{4,8}=\emptyset$ and $\omega^{\text{min}}_{4,8}$ does not exist.
	\end{example}


 We can now control the loop-erased sections that a co-product may extract by admitting only those cuts which are corollas within their relevant temporal context and only if those cuts are contiguous subwalks including the last petals of the corolla:  
 
	\begin{definition}[Admissible cuts]\label{def:admiscut} Let $G$ be a digraph and $\omega=w_1\dots w_\ell\in\mathcal{W}(G)$. A non-empty loop-erased section $\omega^{k,k'}:=w_kw_{k+1}\dots w_{k\rq}\in\text{LES}(\omega)$ is an {\it admissible cut} of $\omega$ when $\omega^{k,k'}\neq \omega$ and either  $\omega^{l,l'}:=\omega^{\text{min}}_{k,k'}$ does not exist or $w_k$ does not appear in $w_{k'+1}\cdots w_{l'}$.
		The set of admissible cuts of $\omega$ is denoted $\text{AdC}(\omega)$.
	\end{definition}

\begin{remark}
The condition that for $\omega^{k,k'}\in\text{LES}(\omega)$, $\omega^{l,l'}:=\omega^{\text{min}}_{k,k'}$ either does not exist or $w_l$ does not appear in $w_{k'+1}\cdots w_{l'}$ implies that admissible cuts can only be made right to left in the walk, that is from the latest to the earliest, in \emph{reverse chronological order}. 
\end{remark}

\begin{example}
		In the complete graph $K_5$, consider the walk   
  
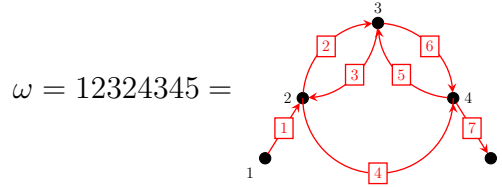
\vspace{-2mm}
  \[\omega=12324345=\WalkGirl\] 
		The loop-erased sections $\omega^{1,3}=232\in\text{LES}(\omega)$ and $\omega^{4,6}=434\in\text{LES}(\omega)$ are both admissible cuts of $\omega$. At the opposite, $\omega^{2,5}=3243\notin \text{LES}(\omega)$
  and so is not an admissible cut.
  \end{example}
	
  \begin{example}
  In the walk \[\omega=12131=\WalkCorollaTwoLoopsThreeVertices{{1,2,3}}{{1/2,3/4}}{}\]
 subwalk $\omega^{2,4}\in\text{LES}(\omega)$ is an admissible cut of $\omega$, while $\omega^{0,2}=121\in\text{LES}(\omega)$	is not  admissible because vertex $1$ is visited again by $\omega_{0,2}^{\text{min}}$ after completion of $\omega^{0,2}$.
	\end{example}

	The notion of admissible cut is well defined because the property of being admissible does not depend on the order in which admissible cuts are considered and removed from the original walk. 
 In particular, if a loop-erased section is an admissible cut of an admissible cut of a walk or of its remainder, then it is an admissible of that walk and vice-versa. This is significant because it indicates that, in spite of the strong chronological constraints created by Lawler's process, cutting out admissible cuts does not alter the other cuts relevant temporal context and thence, their admissibility:
 
 
	\begin{proposition}\label{LemmaCut}
	Let $G$ be a digraph and $\omega\in\mathcal{W}(G)$. 
 \begin{enumerate}[leftmargin=1.8cm, label=\emph{\textbf{Case \arabic*.}}, ref=\arabic*]
	\item\label{CasMickey} If $k<k'<l<l'$ or $l<l'<k<k'$ then,
    \end{enumerate}
	\[
    \omega^{k,k'}\in\emph{AdC}(\omega)\text{ and }\,\omega^{l,l'}\in \emph{AdC}(\omega_{k,k'})\iff \omega^{l,l'}\in\emph{AdC}(\omega)\text{ and }\,\omega^{k,k'}\in \emph{AdC}(\omega_{l,l'}).
    \]
   \begin{enumerate}[leftmargin=1.8cm, label=\emph{\textbf{Case \arabic*.}}, ref=\arabic*]
   \addtocounter{enumi}{1}
	\item\label{CasImbrique} If $k<l<l'\leq k'$ then,
   \end{enumerate}
	\[
   \omega^{k,k'}\in \emph{AdC}(\omega)\text{ and }\,\omega^{l,l'}\in \emph{AdC}(\omega^{k,k'})\iff \omega^{l,l'}\in \emph{AdC}(\omega)\text{ and }\,\omega_{l,l'}^{k,k'}\in \emph{AdC}(\omega_{l,l'}).
    \]
	\end{proposition}
	\begin{proof} 
		\textbf{Case \ref{CasMickey}.}
			We assume $k<k'<l<l'$ without loss of generality, pictorially this is the situation where 
			
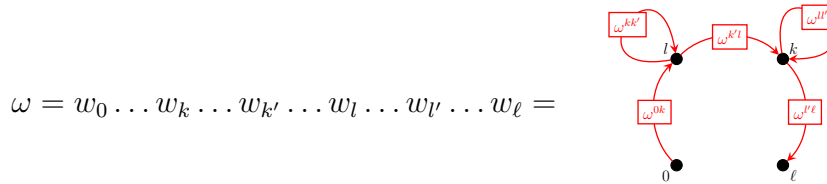
\vspace{-8mm}
            \[
                \omega=w_0\dots w_k\dots w_{k'}\dots w_l\dots w_{l'}\hdots w_\ell =\OpenWalkMickey
            \]
			Suppose that $\omega^{k,k'}\in \text{AdC}(\omega)$ and $\omega^{l,l'}\in \text{AdC}(\omega_{k,k'})$, we first establish that $\omega^{l,l'}\in \text{AdC}(\omega)$.\\[-.5em] 
   
           Given that $\omega^{k,k'}\in\text{LES}(\omega)$, a closed subwalk is erased from $\omega$ if and only if it is either erased from inside of the $\omega^{k,k'}$ section or from outside of it, i.e. $\omega_{k,k'}$. This is because, by Lemma~\ref{PropLEWLES}, erased sections cannot straddle over one-another owing to their step-by-step erasure in chronological order. 
           Here, $\omega^{l,l'}\in \text{AdC}(\omega_{k,k'})$ and since $\omega_{k,k'}\in\text{LES}(\omega)$, then $\omega^{l,l'}$ is an erased closed subwalk of within an erased closed subwalk of $\omega$. This indicates that $\omega^{l,l'}\in\text{LES}(\omega)$.\\[-.5em] 

            To show that $\omega^{l,l'}$ is an admissible loop-erased section of $\omega$, consider the set  $\text{LES}(\omega)_{l,l'}^<$ of loop-erased sections of $\omega$ which strictly include $\omega^{l,l'}$ as subwalk.
            If $\omega^{\text{min}}_{l,l'}$ does not exist, then $\omega^{l,l'}$ is admissible, $\omega^{l,l'}\in\text{AdC}(\omega)$. Suppose instead that $\omega^{m,m'}:=\omega^{\text{min}}_{l,l'}$ exists and recall that $\omega^{l,l'}\in\text{AdC}(\omega_{k,k'})$. This implies one of the two following possibilities:
            \begin{enumerate}
\item[i)] $(\omega_{k,k'})^{\text{min}}_{l,l'}$ does not exist then, 
					\begin{enumerate}
						\item either $m\in \{1,\dots,k\}\cup \{k',\dots,l\}$ and thus $\omega_{k,k'}^{m,m'}\in\text{LES}(\omega_{k,k'})^{<}_{l,l'}\neq \emptyset$ so its minimum exists, a contradiction;
						\item or $m\in\{k+1,\dots,k'-1\}$ but then $\omega^{k,k'}\in\text{AdC}(\omega)\Rightarrow\omega^{m,m'}\notin \text{LES}(\omega)$, a contradiction.
					\end{enumerate}	
					\item[ii)] $\omega^{n,n'}:=(\omega_{k,k'})^{\text{min}}_{l,l'}$ exists, $n\leq l< l'< n'$, and
					\begin{enumerate}
						\item if $k<n <k'$ then $\omega^{k,k'}\in \text{AdC}(\omega)\Rightarrow \omega^{n,n'}\notin\text{LES}(\omega_{k,k'})$ a contradiction;
						\item if $n\geq k'$ then $\omega^{\text{min}}_{l,l'}=(\omega_{k,k'})^{\text{min}}_{l,l'}$ so $\omega^{l,l'}\in\text{AdC}(\omega_{k,k'})\Rightarrow \omega^{l,l'}\in\text{AdC}(\omega)$; 
						\item if $n\leq k$ then $\omega^{n,n'}\in\text{LES}(\omega)^{<}_{l,l'}$ i.e. either $\omega^{m,m'}$ is a subwalk of $\omega^{n,n'}$ or the two are the same and in both cases $\omega^{l,l'}\in\text{AdC}(\omega_{k,k'})\Rightarrow \omega^{l,l'}\in\text{AdC}(\omega)$.
					\end{enumerate}
				\end{enumerate}
This shows that $\omega^{l,l'}\in\text{AdC}(\omega)$.
     Second, we establish that $\omega^{k,k'}\in \text{AdC}(\omega_{l,l'})$.\\[-.5em] 

Since $\omega^{k,k'}\in\text{AdC}(\omega)$ and given that $k'<l$ implies that $\omega^{k,k'}$ is an erased closed subwalk from outside of $\omega^{l,l'}$, then $\omega^{k,k'}\in\text{LES}(\omega_{l,l'})$. Supposing that $\omega^{\text{min}}_{k,k'}$ does not exist then $(\omega_{l,l'})^{\text{min}}_{k,k'}$ does not exist either and $\omega^{k,k'}\in\text{AdC}(\omega_{l,l'})$.  Now suppose instead that $\omega^{\text{min}}_{k,k'}$ exists, then $(\omega_{l,l'})^{\text{min}}_{k,k'}$ is either identical to $\omega^{\text{min}}_{k,k'}$ or is a subwalk of it. In both situations $\omega^{k,k'}\in\text{AdC}(\omega)$ then entails that vertex $w_k=w_k'$ is not visited again after step $k'$ in $\omega^{\text{min}}_{k,k'}$ and its subwalks, hence $\omega^{k,k'}\in\text{AdC}(\omega_{l,l'})$.\\[-.5em]

			Conversely, assuming that $\omega^{l,l'}\in \text{AdC}(\omega)$ and $\omega^{k,k'}\in \text{AdC}(\omega_{l,l'})$ and proceeding as above we obtain that $\omega^{k,k'}\in \text{AdC}(\omega)$ and $\omega^{l,l'}\in \text{AdC}(\omega_{k,k'})$, which proves Case 1 of the Proposition.\\

   \textbf{Case \ref{CasImbrique}.} $k<l<l'\leq k'$, pictorially this is the situation where,
			\[\omega=w_0\dots w_k\dots w_l \dots w_{l'}\dots w_{k'}\dots w_\ell=\OpenWalkLadder,\]
			or 
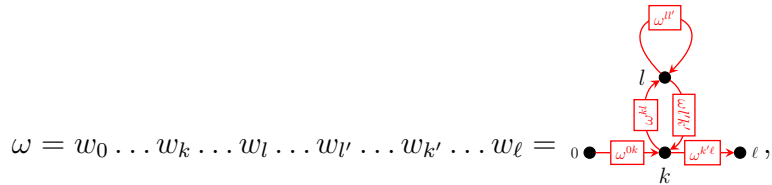
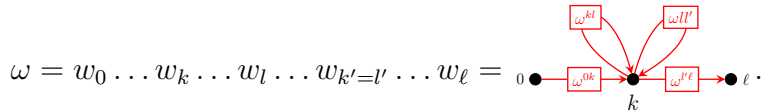
			\[\omega=w_0\dots w_k\dots w_{l} \dots w_{k'=l'}\dots w_\ell=\OpenWalkCorolla.\]
			We assume that $\omega^{l,l'}\in \text{AdC}(\omega)$ and $\omega_{l,l'}^{k,k'}\in \text{AdC}(\omega_{l,l'})$ and first establish that $\omega^{k,k'}\in \text{AdC}(\omega)$.\\[-.5em]  
            
            Since $\omega_{l,l'}^{k,k'}\in \text{LES}(\omega_{l,l'})$, by Lemma~\ref{PropLEWLES} there exists no pair of integers $m,m'$ with $0\leq m\leq k\leq m'\leq k'\leq \ell$ and $w_m=w_{m'}\in V(\text{LEW}_k(\omega_{l,l'}))$. Additionally, $k<l<l'\leq k'$ implies that $\text{LEW}_{k}(\omega_{l,l'})=\text{LEW}_{k}(\omega)$. Replacing the former with the latter gets $w_{m}\in V(\text{LEW}_k(\omega))$ which entails, by Lemma~\ref{PropLEWLES}, that $\omega^{k,k'}\in\text{LES}(\omega)$. 
            It remains to verify that $\omega^{k,k'}$ is admissible. To this end remark that $
            (\omega_{l,l'})^{\text{min}}_{k,k'}$ exists if and only if $\omega^{\text{min}}_{k,k'}$ exists since $k<l<l'\leq k'$ and, by assumption, $\omega^{l,l'}\in \text{AdC}(\omega)$. In this case $(\omega_{l,l'})^{\text{min}}_{k,k'}=(\omega_{k,k'})^{\text{min}}_{l,l'}$ so $\omega^{k,k'}\in \text{AdC}(\omega)$.\\[-.5em]  

    Second we show that $\omega^{l,l'}\in \text{AdC}(\omega^{k,k'})$. 
    By assumption $\omega_{l,l'}\in \text{AdC}(\omega)$, then $\omega^{l,l'}$ is a closed erased subwalk of $\omega$ and, by $k<l<l'\leq k'$, it is more precisely a subwalk of $\omega^{k,k'}$. Then $\omega^{l,l'}\in \text{LES}(\omega^{k,k'})$.

    To verify that the loop-erased section $\omega^{l,l'}$ is admissible in $\omega^{k,k'}$,  observe that assumption $\omega^{k,k'}_{l,l'}\in \text{AdC}(\omega_{l,l'})$ implies, depending on the situation:
					\begin{enumerate}
						\item[i)] if $l'<k'$ then $\omega^{k,k'}\in\text{LES}(\omega^{k,k'})^{<}_{l,l'}$ and so \((\omega^{k,k'})^{\text{min}}_{l,l'}=\omega^{\text{min}}_{l,l'}\);
						\item[ii)] if $l'=k'$ then $\text{LES}(\omega^{k,k'})^{<}_{l,l'}$ is empty and  $(\omega^{k,k'})^{\text{min}}_{l,l'}$ does not exist.
					\end{enumerate}
					Therefore in both situations $\omega^{l,l'}$ is an admissible cut of $\omega^{k,k'}$.\\[-.5em]

			The converse results, namely proving that $\omega^{l,l'}\in \text{AdC}(\omega)$ and $\omega_{l,l'}^{k,k'}\in \text{AdC}(\omega_{l,l'})$ while assuming $\omega^{k,k'}\in \text{AdC}(\omega)$ and $\omega^{l,l'}\in\text{AdC}(\omega^{k,k'})$ are obtained completely similarly, yielding Case 2 of the Proposition.
	\end{proof}

\subsection{All walks are totally-ordered temporal trees}
Lemma~\ref{PropLEWLES} and Proposition~\ref{LemmaCut} strongly suggest that any walk on any graph is chronologically equivalent to a tree where the root node is the self-avoiding skeleton of the walk and each non-root node stands for a simple cycle, see Theorem~\ref{WalkTree} below. 
In that tree, Lawler's procedure erases nodes from the leaves down to the root and operates on the branches from left to right (or more precisely along the direction given to time). 
That is, time totally orders the walk's tree structure. 
Formally, this translates into a total order on the set of admissible cuts of a walk:

\begin{definition}[Time-ordering of the admissible cuts]\label{Order}
		Let $G$ be a digraph and $\omega\in\mathcal{W}(G)$. Assuming that $\text{AdC}(\omega)\neq\emptyset$ we define the relation $\leqslant_{\clk}$ on $\text{AdC}(\omega)$ as follows:
		$$
		\omega^{k,k'},\omega^{l,l'}\in\text{AdC}(\omega):~\omega^{k,k'}\leqslant_{\clk}\omega^{l,l'}\iff l\leq k<k'\leq l'\,\text{ or }\,k<k'<l<l'.
		$$
  That is, $\omega^{k,k'}\leqslant_{\clk} \omega^{l,l'}$ if and only if either $\omega^{k,k'}$ is erased prior to $\omega^{l,l'}$ or, if both are erased simultaneously, $\omega^{k,k'}$ began after $\omega^{l,l'}$.
	\end{definition}

	\begin{example}\label{OrderEx}
		Consider the walk, \[\omega=34555444678879=\WalkExampleOrder{}{{3,...,9}}\]
		and three of its admissible cuts, $\omega^{2,4}=555$, $\omega^{3,4}=55$ and $\omega^{10,11}=88$. Then \(\omega^{3,4}\leqslant_{\clk} \omega^{2,4}\leqslant_{\clk} \omega^{10,11}\).
	\end{example}
	
	\begin{proposition}\label{TotalOrder}
		Let $G$ be a digraph, $\omega\in\mathcal{W}(G)$ and suppose that $\text{AdC}(\omega)\neq\emptyset$. Then $\text{AdC}(\omega)$ is totally ordered by $\leqslant_{\clk}$.
	\end{proposition}
	
	\begin{proof}
		By contradiction. Suppose that there exists two admissible cuts of $\omega$, $\omega^{k,k'}$ and $\omega^{l,l'}$ such that $k<l'$ wlog, and which cannot be time-ordered. Then necessarily $k<l\leq k'\leq l'$. But since both $\omega^{k,k'}$ and $\omega^{l,l'}$ are admissible they are loop-erased sections of $\omega$. But by Lemma~\ref{PropLEWLES} they may not straddle, that is we may not have $k<l\leq k'\leq l'$, a contradiction. Transitivity, reflexivity and anti-symmetry of $\leqslant$ follow immediately.
	\end{proof}

	\begin{example}\label{CutOrderEx}
 Consider again the walk of Example~\ref{OrderEx}. It admissible cuts set is $$\text{AdC}(\omega)=\{\omega^{1,7},\omega^{2,4},\omega^{3,4},\omega^{5,7},\omega^{6,7},\omega^{9,12},\omega^{10,11}\},$$ and the total order on it is
		\[\omega^{3,4}\leqslant_{\clk}\omega^{2,4}\leqslant_{\clk}\omega^{6,7}\leqslant_{\clk} \omega^{5,7}\leqslant_{\clk}\omega^{1,7}\leqslant_{\clk}\omega^{10,11}\leqslant_{\clk} \omega^{9,12}.\]
	\end{example}

\begin{theorem}[All walks are temporal-trees]\label{WalkTree}
Let $G$ be a digraph and $\omega\in\mathcal{W}(G)$. Then $\omega$ has the temporal-structure of a tree $t(\omega)$ whose nodes are totally ordered by $\leqslant_{\clk}$ according to a depth-first order. 
\end{theorem}

\begin{proof}
   To establish the theorem, we first map walks to cacti  then cacti to trees:
   \begin{definition}[Cactus]\label{DefCac}
		Let $G$ be a digraph. A walk $\omega=w_0\dots w_\ell\in \mathcal{W}(G)$ is a cactus if and only if for any  $0\leq k<k'\leq \ell$, $w_k=w_{k'}\iff \omega^{k,k'}\in \text{LES}(\omega)$. We denote  $\mathcal{C}$ the set of all cacti on the complete graph \(K_\mathbb{N}\) with $V(K_\mathbb{N})=\mathbb{N}$ and by $\text{Cact}(G)$ the vector space spanned by the cacti on $G$.
	\end{definition}
    In a cactus all repeated vertices delimit valid loop-erased sections, which means that there cannot be patterns such as $\omega=1\underline{2}1\underline{2}1$ as $212\notin\text{LES}(\omega)$. Intuitively a cactus is therefore a `disentangled' walk, where every instance of repeated vertex is the root of a simple-cycle erased by Lawler's procedure. 
    We may always map walks to cacti by defining, 
   \begin{align}
   C:\,\mathcal{W}(G)&\to \mathcal{C},\label{CactusMapDef}\\
   \omega=w_0\cdots w_\ell&\mapsto \kappa:=C(\omega)=c_0\cdots c_\ell,\nonumber
   \end{align}
   where $\kappa$ is the cactus defined as follows: $c_0=w_0$ and for any $k\in\{0,...,\ell-1\}$ if $\text{LEW}_{k+1}(\omega)=\text{LEW}_k(\omega)w_{k+1}$ then \(c_{k+1}=\max(V(c_0\dots c_k))+1\); else  $c_{k+1}=c_l$ where \(l=\max(i\in\{0,\dots,k\},~w_l=w_{k+1})\). 
   
   In words, considering the loop-erased walk $\text{LEW}_k(\omega)$ at step $k$, if  vertex $w_{k+1}$ reached at step $k+1$ is distinct from those of $\text{LEW}_k(\omega)$ then $c_{k+1}$ is a vertex with integer label given by the length of $\text{LEW}_k(\omega)$ plus one (an expedient ensuring that we map distinct labels to distinct labels). If instead vertex $w_{k+1}$ was visited at some step $l$ prior to step $k+1$ in the loop-erased walk, that is $w_l\cdots w_{k+1}$ closes an erased-section, then $c_{k+1}$ is given the same label as $c_l$. For example walk $\omega=12121$ becomes $\kappa:=C(\omega)=12131$. Because the new labels are not labels of nodes of $G$, it may be that $\kappa$ is not a valid walk on $G$ but at least it is  a walk on a complete graph $K_{\mathbb{N}}$, see \S\ref{Cmorphism} below.
   This is sufficient for our purpose: by definition $\omega$ and $\kappa:=C(\omega)$ have the same length,  $\kappa$ is a cactus, and $\omega$ and $\kappa$ share the same temporal structure,
   \begin{equation}\label{finjectAdC}
   \omega^{k,k'}\in \text{AdC}(\omega)\iff \kappa^{k,k'}\in \text{AdC}\big(\kappa\big).
   \end{equation}
Since any two simple cycles in cacti share at most one vertex (their roots), we define a tree $t(\omega)$ from $\kappa$ by drawing a tree-node $t$ for every simple cycle $\sigma$ of $\kappa$. Two nodes $t$ and $t'$ of the tree are connected if and only if the corresponding simple cycles $\sigma$ and $\sigma'$ share a vertex in $\kappa$. Finally we add a root node representing the (possibly trivial) self-avoiding skeleton of $\omega$. The time-order $\leqslant_{\clk}$ now totally orders the nodes of $t(\omega)$: for $t,t'$ two nodes of $t(\omega)$ corresponding to simple cycles $\sigma$ and $\sigma'$ of $\kappa$, $t\leqslant_{\clk} t' \iff \sigma$ is erased prior to $\sigma$ in $\kappa$. This builds a reverse depth-first order on the nodes of $t(\omega)$ with the top-left leaf of the tree corresponding to the first erased simple cycle, hence the smallest as per $\leqslant_{\clk}$.
\end{proof}
\begin{example}
As an example, consider the walk $\omega=12332331$. Then $\kappa:=C(\omega)=12332441$ and the tree $\tau:=t(\omega)$ is
$$
  \omega= \WalkNonAssocNonCACT,\quad \kappa = \WalkNonAssocCACT,\quad t(\omega)=\TreeCACT
$$
The simple cycles of $\kappa$ are $1241$, corresponding to node $d$ of the tree; $232$ (node $b$ of the tree); $33$ (node $a$) and $66$ (node $c$). The root node of $t(\omega)$ stands for the trivial walk `1' on vertex 1, which is the self-avoiding skeleton of $\omega$. The time-order on the tree nodes is $a\leqslant_{\clk} b\leqslant_{\clk} c \leqslant_{\clk} d \leqslant \text{Root}.$

Although the tree $t(\omega)$ depends on the walk $\omega$, an universal tree can be constructed for all walks of a given digraph $G$. Considering only the trees $t(\omega)$ obtained from walks with no repeated sections produces a finite number of structurally distinct trees from all walks on $G$. These trees can be ordered partially by inclusion and the resulting poset always admits an unique maximum. This maximum tree is of paramount importance to $G$: it is one of the few invariants of its hike monoid \cite{Fromentin2023, Giscard2017}, and dictates the shape of all branched continued fractions counting walks on $G$ \cite{Giscard2012}. 
	
\end{example}

\section{The co-preLie co-algebra of walks}\label{sectionCoproduct}
\subsection{Co-product}
With the notion of admissible cut, we may now formally define the co-product associated to Lawler's process, by mapping a walk to a sum over all its admissible cuts tensored with their remainders:
	
	\begin{definition}[Co-product]
		Let $G$ be a digraph. The co-product associated to Lawler's process is the linear map $\Delta_{\text{CP}}$ defined by,
		\[\Delta_{\text{CP}} :\left\{\begin{array}{rcl} \WGamma{W}{G} &\rightarrow &\WGamma{W}{G}\otimes \WGamma{W}{G} \\
			\omega &\mapsto& \displaystyle\Delta_{\text{CP}}(\omega)=\sum_{\omega^c\in \text{AdC}(\omega)} \omega_c \otimes \omega^c.\end{array}\right.\]	
	\end{definition}
	
	An essential property of this co-product is that a walk is primitive for it if and only if it is a simple path or a simple cycle. 
	\begin{proposition}
		Let $G$ be a digraph and $\omega\in\mathcal{W}(G)$. Then,
		\[\Delta_{\emph{\text{CP}}}(\omega)=0\iff \omega\in\emph{\text{SAW}}(G)\cup\emph{\text{SAP}}(G).\]
	\end{proposition}
	
	\begin{proof}

 Let $\omega=w_0\dots w_\ell\in\mathcal{W}(G)$. If $\omega\in \text{SAW}(G)$ then it has no cycles and $\text{AdC}(\omega)=\emptyset$ so $\Delta(\omega)=0$. If $\omega\in \text{SAP}(G)$, then $\text{LES}(\omega)=\{\omega\}$, $\text{AdC}(\omega)=\emptyset$ since the walk is not an admissible cut of itself, therefore $\Delta(\omega)=0$. 

        Now suppose that $\omega\notin \text{SAW}(G)\cup\text{SAP}(G) $. Then $\omega$ has at least one simple cycle and we may consider the last such cycle $\omega^{k,k'}\in\text{LES}(\omega)$ erased from $\omega$ by Lawler's process. This simple cycle cannot be $\omega$ itself since $\omega\notin \text{SAP}(G)$. Furthermore after step $k'$ no vertex of $\text{LEW}_{k'}(\omega)$ is visited again in $w_{k'+1}\cdots w_\ell$ as otherwise $\omega^{k,k'}$ would not be the last erased simple cycle. This simply indicates that the last erased simple cycles is not within a wider erased section by virtue of being the last to be removed. Thus $\omega^{\text{min}}_{k,k'}$ does not exist, $\omega^{k,k'}\in \text{AdC}(\omega)$, and $\Delta(\omega)\neq 0$.
	\end{proof}

 \begin{example}\label{DeltaOmega}
 Consider the walk $\omega=1233234441$, then
 $$
 \Delta_{\text{CP}}(1233234441)=123323441\otimes 44 + 12332341 \otimes 444 + 123234441 \otimes 33 + 1234441\otimes 2332,
 $$
 or, graphically,
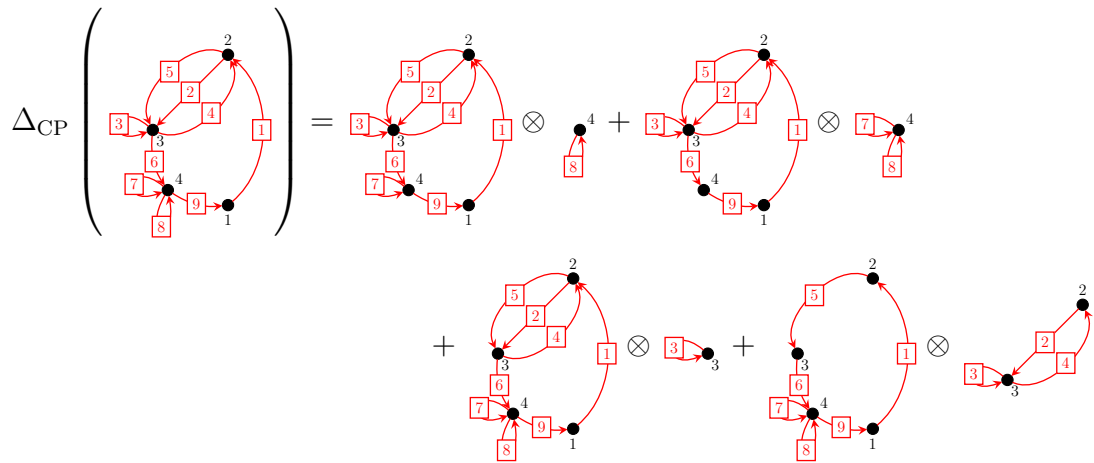
\begin{align*}
&\Delta_{\text{CP}}\left(\WalkNonAssocCP\right)=\WalkNonAssocCPtermA\otimes \WalkNonAssocCPterma +\WalkNonAssocCPtermB\otimes \WalkNonAssocCPtermb\\ &\hspace{55mm}+\WalkNonAssocCPtermC\otimes \WalkNonAssocCPtermc+\WalkNonAssocCPtermD\otimes \WalkNonAssocCPtermd
\end{align*}
   \end{example}

	\subsection{The co-preLie property}
	Having established the definition of the co-product associated to the Lawler process and identified its primitive walks, we now turn to the co-algebraic structure it gives to the walk vector space $\WGamma{W}{\G}$. Recall that:
	\begin{definition}
		A co-preLie co-algebra is a couple $(\mathcal{V},\Delta)$ where $\mathcal{V}$ is a vector space and $\Delta:\mathcal{V}\rightarrow \mathcal{V}\otimes \mathcal{V}$ is a linear map such that for any $v\in \mathcal{V}$ the following relation is satisfied
		\[(\Delta\otimes \Id-\Id\otimes \Delta)\circ \Delta(v)=(\Id\otimes \tau)\circ\left( \Delta\otimes \Id-\Id\otimes \Delta\right)\circ \Delta(v)\]
		where $\Id$ is the identity map and $\tau$ is the twisting linear map, $\tau: \mathcal{V}\otimes \mathcal{V}\to \mathcal{V}\otimes \mathcal{V}$, $\tau(u\otimes v)=v\otimes u$.	
	\end{definition}

	\begin{theorem}\label{DeltaCPTHM}
		The vector space $\WGamma{W}{\G}$, equipped with the coproduct $\Delta_{\emph{\text{CP}}}$, is a co-preLie (but not co-unital) co-algebra.
	\end{theorem}

We present two proofs of this result. The first, given immediately below, is a direct approach based on the properties of admissible cuts. The second proof, presented in \S\ref{BraceSection}, obtains the theorem as a corollary of the Hopf structure on the tensor algebra generated by $\mathcal{W}(G)$ via a brace coalgebra construction.
 
	\begin{proof} Let $\omega=w_0\dots w_\ell\in\mathcal{W}(G)$. We begin with evaluating 
		$
		(\Delta_{\text{CP}}\otimes \Id)\circ \Delta_{\text{CP}}(\omega)
		$
		explicitly.
		To that end consider an admissible cut $\omega^{k,k'}\in\text{AdC}(\omega)$, assuming that  $\omega_{k,k'}$ is not self-avoiding nor a simple cycle as this leads to a 0 result.
        Then, 
		\((\Delta_{\text{CP}}\otimes \Id)(\omega_{k,k'}\otimes \omega^{k,k'})\)
	      yields a sum over cuts that fall into four distinct cases, depending on the second cut's coordinates $l,\,l'$ relatively to $k,\,k'$:
		\begin{itemize}
			\item[1)] $l<l'<k<k'$, i.e. $\omega=w_0\cdots w_l\cdots w_{l'}\cdots w_{k}\cdots w_{k'}\cdots w_\ell$, this gets cut as $\omega_{k,k';l,l'}\otimes \omega^{l,l'}\otimes \omega^{k,k'}$,
			\item[2)] $k<k'<l<l'$, i.e $\omega=w_0\cdots w_k\cdots w_{k'}\cdots w_{l}\cdots w_{l'}\cdots w_\ell$, this gets cut as $\omega_{k,k';l,l'}\otimes \omega^{l,l'}\otimes \omega^{k,k'}$,
			\item[3)] $l<k<k'< l'$, i.e. $\omega=w_0\cdots w_l\cdots w_{k}\cdots w_{k'}\cdots w_{l'}\cdots w_\ell$, this gets cut as $\omega_{l,l'}\otimes \omega_{k,k'}^{l,l'}\otimes \omega^{k,k'}$,
			\item[4)] $l<l'=k<k'$, i.e. $\omega=w_0\cdots w_l\cdots w_{l'=k}\cdots w_{k'}\cdots w_{\ell}$, this gets cut as $\omega_{l,k'}\otimes \omega_{k,k'}^{lk} \otimes \omega^{k,k'}$. Remark that $\omega^{l,l'}$ is not an admissible cut of $\omega$ because $w_{l'}=w_{k'}$ occurs after step  $l'$ in $\omega^{\text{min}}_{l,l'}$.
		\end{itemize}
		By Case~\ref{CasMickey} of Proposition~\ref{LemmaCut}, if an admissible cut falls into situation 1) above, another one will be admissible as per situation 2). Thus,
		\begin{align*}
			&(\Delta_{\text{CP}}\otimes \Id)\circ \Delta_{\text{CP}}(\omega)=\sum_{\substack{c\in \text{AdC}(\omega)\\ \omega_{k,k'}\notin\text{SAW}(\G)\\ \omega_{k,k'}\notin\text{SAP}(\G)}}\sum_{c'\in \text{AdC}(\omega_{k,k'})}\begin{cases}
				\omega_{k,k';l,l'}\otimes \omega^{l,l'}\otimes \omega^{k,k'}&l<l'<k<k' \\
				\omega_{k,k';l,l'}\otimes \omega^{l,l'}\otimes \omega^{k,k'}&k<k'<l<l',\\
				\omega_{l,l'}\otimes \omega_{l,l'}^{k,k'}\otimes \omega^{k,k'}& l<k<k'< l',\\
				\omega_{l,k'}\otimes \omega_{k,k'}^{l,k} \otimes \omega^{k,k'}&l<l'=k<k',                                               
			\end{cases}
		\end{align*}
		where we used $c:=\omega^{k,k'}$ and $c':=\omega^{l,l'}$ to alleviate the notation.\\

       Now we turn to $(\Id\otimes \Delta_{\text{CP}})\circ \Delta_{\text{CP}}(\omega)$. 
		Let again $\omega^{k,k'}\in\text{AdC}(\omega)$ be an admissible cut of $\omega$ which we assume not to be a simple cycle as this would lead to a 0 result. Then,
		\((\Id\otimes \Delta_{\text{CP}})\circ (\omega_{k,k'}\otimes \omega^{k,k'})\)
        yields a sum over cuts that fall into two distinct cases, depending on the second cut's coordinates $l,\,l'$ inside of $\omega_{k,k'}$:
		\begin{itemize}
			\item[1)] $k<l<l'<k'$, that is $\omega=w_0\cdots w_k\cdots w_{l}\cdots w_{l'}\cdots w_{k'}\cdots w_\ell$, which gets cut as $\omega_{k,k'}\otimes \omega^{k,k'}_{l,l'}\otimes \omega^{l,l'}$,
			\item[2)] $k<l<l'=k'$, that is $\omega=w_1\cdots w_k\cdots w_{l}\cdots w_{l'=k'}\cdots w_m$, which gets cut as $\omega_{k,k'}\otimes \omega^{k,k'}_{l,k'}\otimes \omega^{l,k'}$.
		\end{itemize}
		Here we do not need to consider the case $k=l<l'\leq k'$. Indeed, either $k=l<l'= k'$ then $\omega^{l,l'}=\omega^{k,k'}$ meaning we cut $\omega^{k,k'}$ out of itself, which is not admissible; or $k=l<l'< k'$ but then, $\omega^{l,l'}$ is not admissible because $w_{l'}=w_{k'}$ is visited again after step $l'$. Rather, in that situation it is $\omega^{l',k'}$ that is admissible and falls into case 2) above. 
		Thus,
		\begin{align*}
			&(\Id\otimes \Delta_{\text{CP}})\circ \Delta_{\text{CP}}(\omega)=\sum_{\substack{c\in \text{AdC}(\omega)\\ \omega^{k,k'}\notin\text{SAP}(\omega)}}\sum_{c'\in \text{AdC}(\omega^{k,k'})}\begin{cases}
				\omega_{k,k'}\otimes \omega^{k,k'}_{l,l'}\otimes \omega^{l,l'}&k<l<l'<k',\\
				\omega_{k,k'}\otimes \omega^{k,k'}_{l,k'}\otimes \omega^{l,k'}&k<l<l'=k',                                              
			\end{cases}
		\end{align*}
        where we used $c:=\omega^{k,k'}$ and $c':=\omega^{l,l'}$ to alleviate the notation.
		By Case~\ref{CasImbrique} of Proposition~\ref{LemmaCut}, gathering everything, we obtain
		\[(\Delta_{\text{CP}}\otimes \Id-\Id\otimes \Delta_{\text{CP}})\circ \Delta_{\text{CP}}(\omega)=\sum_{\substack{c\in \text{AdC}(\omega)\\ \omega_{k,k'}\notin\text{SAW}(\G)\\ \omega_{k,k'}\notin\text{SAP}(\G)}}\sum_{\substack{c'\in \text{AdC}(\omega_{k,k'})\\l<l'<k<k'\\k<k'<l<l'}}
		\omega_{k,k';l,l'}\otimes \omega^{l,l'}\otimes \omega^{k,k'}.\]
  Remark how $k,k'$ and $l,l'$ now play completely symmetric roles in the above so that the co-prelie relation holds for all walks $\omega\in\mathcal{W}(G)$,
		\begin{align*}
			&(\Delta_{\text{CP}}\otimes \Id-\Id\otimes \Delta_{\text{CP}})\circ \Delta_{\text{CP}}(\omega)=(\Id\otimes \tau)\circ\left( \Delta_{\text{CP}}\otimes \Id-\Id\otimes \Delta_{\text{CP}}\right)\circ \Delta_{\text{CP}}(\omega).
		\end{align*}
	This indicates, perhaps suprisingly, that Lawler's intuitive chronological removal of the simple cycles from walks naturally endows their vector space with a sophisticated co-preLie structure. 
	\end{proof}
 \begin{example}\label{CopreLieEx}
     Consider again the walk $\omega=1233234441$ of Example~\ref{DeltaOmega}. Then,
\begin{align*}
 \big(\Delta_{\text{CP}}\otimes \text{Id}\big)\circ\Delta_{\text{CP}}(1233234441)=&~12332341\otimes 44\otimes 44 +12323441\otimes 33\otimes 44+123441\otimes 2332\otimes 44\\
 &+ 1232341\otimes 33 \otimes 444 +12341 \otimes 2332 \otimes 444\\
 &+12323441\otimes 44 \otimes 33 +1232341\otimes 444 \otimes 33 +1234441\otimes 232 \otimes 33\\
 &+ 123441\otimes 44\otimes 2332+12341\otimes 444\otimes 2332,\\
 \big(\text{Id}\otimes \Delta_{\text{CP}}\big)\circ\Delta_{\text{CP}}(1233234441)=&~12332341 \otimes 44\otimes 44  + 1234441\otimes 232\otimes 33.
\end{align*}
From this, reordering the terms for convenience, we obtain
 \begin{align*}
 \big(\Delta_{\text{CP}}\otimes \text{Id}-\text{Id}\otimes \Delta_{\text{CP}}\big)\circ\Delta_{\text{CP}}(1233234441)=&~12323441\otimes 33\otimes 44+12323441\otimes 44 \otimes 33\\ 
 &+ 123441\otimes 2332\otimes 44+ 123441\otimes 44\otimes 2332\\
 &+ 1232341\otimes 33 \otimes 444 ++1232341\otimes 444 \otimes 33\\
 &+12341 \otimes 2332 \otimes 444+12341\otimes 444\otimes 2332,
\end{align*}
which is invariant under the action of $\text{Id}\otimes \tau$ as dictated by Theorem~\ref{DeltaCPTHM}.
\end{example}

	\section{Hopf structures on the tensor and symmetric algebras of walks}\label{HopfTensor}

	As in the case of trees due to Connes and Kreimer in \cite{Connes1999} or the case of decorated trees due to Foissy \cite{Foissy2002a}, for any walk, we can extend the notion of admissible cut to allow for multiple simultaneous cuts, which we call \textit{extended} admissible cuts. Thanks to this construction, a dual of Oudon and Guin's own \cite{Oudom2008}, we get a co-product compatible with the tensor and symmetric algebra structures generated by walks on \(\G\). 

\subsection{Extended admissible cuts}
 We begin by defining the notion of extended admissible cuts of a walk, then show that they turn the tensor and symmetric algebras of walks into Hopf algebras. Finally, we present three special families of walks, ladders, corollas and cacti, and explain how we can make them into Hopf algebras.
	\begin{definition}
		Let $\G$ be a finite connected non-empty graph and $\K$ a field of characteristic $0$. 
		\begin{enumerate}
			\item We define $\Tens{\WGamma{W}{\G}}$ as the tensor algebra generated by $\WGamma{W}{\G}$. To alleviate the notation, for walks $\omega_1, \dots, \omega_p\in\mathcal{W}(G)$, the tensor \(\omega_1\otimes\dots\otimes\omega_p\) will be denoted by \(\omega_1\,|\,\dots\,|\,\omega_p\). Such elements of $\Tens{\WGamma{W}{\G}}$ are called forests. 
			\item Let $\omega=w_0\dots w_\ell$ be a walk in $\G$. In keeping we common terminology for Hopf algebras  we call \textit{degree} $\deg(\omega)$ the length of the walk $\omega$.
		\end{enumerate}
	\end{definition}
	
	We recall that the tensor algebra $\Tens{\WGamma{W}{\G}}$ is equipped with the concatenation product $\bigcdot$,
	\[\bigcdot :\left\{\begin{array}{rcl}\Tens{\WGamma{W}{\G}} \otimes \Tens{\WGamma{W}{\G}} &\longrightarrow &\Tens{\WGamma{W}{\G}} \\
		\omega_1\,|\,\dots\,|\,\omega_m\otimes \omega'_1\,|\,\dots\,|\,\omega'_n &\longmapsto& \omega_1\,|\,\dots\,|\,\omega_m\,|\,\omega'_1\,|\,\dots\,|\,\omega'_n,\end{array}\right.\]
	and the degree $\deg(\omega_1|\dots|\omega_m)=\displaystyle\sum_{i=1}^{m}\deg(\omega_i)$ is the sum of the involved walks' degrees. By construction, $(\Tens{\WGamma{W}{\G}},\bigcdot)$ is an unital associative algebra with unit the empty forest $()$, identified with $\mathbf{1}\in\K$, written in bold font so as to distinguish it from a vertex label `1'.\\[-.5em]

	

	\begin{definition}[Extended admissible cut]\label{EAdCDef} Let $\G$ be a digraph and $\omega\in\mathcal{W}(G)$. An {\it extended admissible cut} of $\omega$ is the tensor product of $n\in\mathbb{N}\backslash\{0\}$ consecutive admissible cuts $\omega^{k_i,k'_i}\in\text{AdC}(\omega)$ which are \emph{non-overlapping} in $\omega$, that is $k_1<k_1'<k_2<k_2'<\dots <k_n<k_n'$. We write,
 $$
 \omega^{k_1,k_1';\dots; k_n,k_n'}:=\omega^{k_1,k_1'}\,|\,\dots\,|\,\omega^{k_n,k_n'}\in \Tens{\WGamma{W}{\G}}. 
 $$
The set of extended admissible cuts of $\omega$ is denoted $E\text{AdC}(\omega)$. Observe that $\text{AdC}(\omega)\subset E\text{AdC}(\omega)$.
	\end{definition}

To alleviate the notation whenever possible we designate an extended admissible cut by a single letter, e.g. $\omega^c\in E\text{AdC}(\omega)$ and might then simply write that $c$ is an extended admissible cut of $\omega$.

 \begin{example}\label{EAdCExample}Consider the walk $\omega=123324441$, then \[E\text{AdC}(\omega)=\{33,44,444,2332, 33\,|\,44, 33\,|\,444,2332\,|\,44,2332\,|\,444\}.\] 
	\end{example}

\begin{remark}
The notion of extended admissible cut generalizes straightforwardly from $\mathcal{W}(G)$ to $\Tens{\WGamma{W}{\G}}$. Consider  $\omega_1\,|\,\dots\,|\,\omega_m\in \Tens{\WGamma{W}{\G}}$ with $\omega_i\in \mathcal{W}(G)$, then an extended admissible cut of this is an element $\omega^{c_i}\,|\,\dots\,|\,\omega^{c_m}\in\Tens{\WGamma{W}{\G}}$
 such that all $\omega^{c_i}\in E\text{AdC}(\omega_i)$.
 \end{remark}
 
 As stated Definition~\ref{EAdCDef} the admissible cuts constituting an extended admissible cut $\omega^c$ of a walk $\omega=w_0\cdots w_\ell$ are non-overlapping, i.e. $\omega^c:=\omega^{k_1,k_1';\dots; k_n,k_n'}\in E\text{AdC}(\omega)$ satisfies
   \(0\leq k_1<k_1'<k_2<k_2'<\dots <k_n<k_n' \leq \ell.\)
  We can therefore meaningfully denote
   \[\omega_c:=\omega_{k_1,k_1';\dots;k_n,k_n'}=w_0\dots w_{k_1} w_{k_1' +1}\dots w_{k_2} w_{k_2' +1} \dots w_{k_n} w_{k_n' +1}\dots w_\ell,\] 
   for what remains of $\omega$ after erasure of all $\omega^{k_i,k_i'}$. Since admissible cuts are closed subwalks of a walk, $\omega_{c}$ is still a walk. Together with the non-overlapping condition this implies that, for any $1\leq i\leq n$, 
   \begin{equation}\label{EAdC2AdC}
   \omega^{k_1,k_1';\dots; k_n,k_n'}\in E\text{AdC}(\omega)\Rightarrow \omega^{k_i,k_i'}\in \text{AdC}(\omega_{k_1,k_1';\dots;k_{i-1},k_{i-1}';k_{i+1},k_{i+1}';\dots; k_n,k_n'}).
   \end{equation}
    Extended admissible cuts are `well behaved' in the sense that such cuts and their remainders satisfy an analog of Proposition~\ref{LemmaCut} for admissible cuts: 

 \begin{proposition}
     \label{LemmaHopf} Let $\G$ a digraph, $\omega\in\mathcal{W}(G)$, $\omega^c\in E\mathrm{AdC}(\omega)$. Then,
$$
   \omega^{c'}\in\mathrm{AdC}(\omega_{c})\Rightarrow \omega^{c'}\in\mathrm{AdC}(\omega)
   $$
which also implies $\omega^{c'}\in E\mathrm{AdC}(\omega_{c})\Rightarrow \omega^{c'}\in E\mathrm{AdC}(\omega)$. Furthermore,
$$
   \omega^{c'}\in E\mathrm{AdC}(\omega^{c})\Rightarrow \omega^{c'}\in E\mathrm{AdC}(\omega)
   $$
   	\end{proposition}

	\begin{proof}
 Let  $\omega^{c}:=\omega^{k_1,k_1';\dots ;k_n,k_{n}'}$ be an extended admissible cut of $\omega$. By virtue of Proposition~\ref{LemmaCut}, an admissible cut of the remainder of an admissible cut of a walk  is an admissible cut of that walk, 
 \[\omega^{k,k'}\in\text{AdC}(\omega),~ \omega^{l,l'}\in\text{AdC}(\omega_{k,k'})\Rightarrow \omega^{l,l'}\in\text{AdC}(\omega).\] 
  In addition $\omega_{k_2,k_2';\dots;k_n,k_n'}^{k_1,k_1'}\in \text{AdC}(\omega_{k_2,k_2';\dots;k_n,k_n'})$ by virtue of the fact that extended admissible cuts only comprise non-overlapping admissible cuts. Therefore we get
  \[\omega^{l,l'}\in\text{AdC}(\omega_{k_1,k_1';\dots;k_n,k_n'})\Rightarrow \omega^{l,l'}\in\text{AdC}(\omega_{k_2,k_2';\dots;k_n,k_n'}).\]
 Iterating this observation leads to the first claim for $\omega^{c'}:=\omega^{l,l'}$. Note that since all cuts are non-overlapping we could have chosen to remove the $k_i,k_i'$ cuts in any order in the iteration. The result for extended admissible cuts is now immediate since such cuts comprise only non-overlapping admissible cuts to each of which  we apply the result just proven.

For the second claim, observe that by Case~\ref{CasImbrique} of Proposition~\ref{LemmaCut}, $\omega^c\in\text{AdC}(\omega)$ and $\omega^{c'}\in \text{AdC}(\omega^c)$ implies $\omega^{c'}\in \text{AdC}(\omega)$. The result for extended admissible cuts follows once more from the observation that such cuts comprise only non-overlapping admissible cuts, each of which behaves as dictated by Case~\ref{CasImbrique} of Proposition~\ref{LemmaCut}.
\end{proof}

\subsection{Hopf algebra on walks}\label{HopfTensor}
 We may now define a co-product on $\Tens{\WGamma{W}{\G}}$ by relying on extended admissible cuts and their remainders:
 	\begin{definition}[Extended co-product]\label{DefHopfCoproduct} Let $\G$ be a digraph. Consider the morphism of algebras $\Delta_{\text{H}}$ defined by:
		\[\Delta_{\text{H}} :\left\{\begin{array}{rcl} \Tens{\WGamma{W}{\G}} &\longrightarrow &\Tens{\WGamma{W}{\G}}\otimes \Tens{\WGamma{W}{\G}}\\
			\omega &\longmapsto& \displaystyle\Delta_{\text{H}}(\omega)=\mathbf{1}\otimes\omega+\omega\otimes\mathbf{1}+\displaystyle\sum_{c\in E\text{AdC}(\omega)} \omega_{c} \otimes \omega^{c},
		\end{array}\right.\]
		where the sum runs over all extended admissible cuts $\omega^c$ of $\omega$.
	\end{definition}

	\begin{theorem}\label{HopfT} Let $\G$ a digraph and consider the triple $\mathcal{H}_{\mathcal{T}}:=(\Tens{\WGamma{W}{\G}},\bigcdot,\Delta_{\emph{\text{H}}})$. Equipped with the map $\deg$, it defines a graded connected Hopf algebra.
	\end{theorem}

 
	\begin{proof}
		Observe first that $\deg$ is a graduation by direct calculation. Second, to prove the theorem we must establish that $\Delta_{\text{H}}$ is coassociative. Since $\Delta_{\text{H}}$ is a morphism of algebras it is sufficient to show that for any walk $\omega\in\mathcal{W}(\G)$,
  \[(\Delta_{\text{H}}\otimes \Id)\circ \Delta_{\text{H}}(\omega)=(\Id\otimes\Delta_{\text{H}})\circ \Delta_{\text{H}}(\omega).\]
		Let $\omega$ be a walk in $\G$. Then
		\begin{align*}
			(\Delta_{\text{H}}\otimes \Id)\circ \Delta_{\text{H}}(\omega)&=\omega\otimes \mathbf{1}\otimes \mathbf{1}+\mathbf{1}\otimes \omega\otimes \mathbf{1}+\mathbf{1}\otimes \mathbf{1}\otimes \omega+\sum_{c\in E\text{AdC}(\omega)}\omega_c\otimes \omega^c\otimes \mathbf{1} \\
			&\hspace{-5mm}+\sum_{c\in E\text{AdC}(\omega)}\omega_c\otimes \mathbf{1}\otimes \omega^c  +\sum_{c\in E\text{AdC}(\omega)}\mathbf{1}\otimes \omega_c\otimes \omega^c+\sum_{c\in E\text{AdC}(\omega)}\sum_{c'\in E\text{AdC}(\omega_c)}(\omega_{c})_{c'}\otimes (\omega_c)^{c'}\otimes \omega^{c},
		\end{align*}
		Similarly, 
		\begin{align*}
			(\Id\otimes\Delta_{\text{H}})\circ \Delta_{\text{H}}(\omega)&=\omega\otimes \mathbf{1}\otimes \mathbf{1}+\mathbf{1}\otimes \omega\otimes \mathbf{1}+\mathbf{1}\otimes \mathbf{1}\otimes \omega+\sum_{c\in E\text{AdC}(\omega)}\mathbf{1}\otimes \omega_c\otimes \omega^c\\
   &\hspace{-5mm}+\sum_{c\in E\text{AdC}(\omega)}\omega_c\otimes \mathbf{1}\otimes \omega^c +\sum_{c\in E\text{AdC}(\omega)}\omega_c\otimes \omega^c\otimes \mathbf{1}+\sum_{c\in E\text{AdC}(\omega)}\sum_{c'\in E\text{AdC}(\omega^c)}\omega_{c}\otimes (\omega^{c})_{c'}\otimes (\omega^{c})^{c'},
		\end{align*} 
		So the theorem follows if we prove that 
  \begin{equation}\label{HopfProofEq}
  \sum_{c\in E\text{AdC}(\omega)}\sum_{c'\in E\text{AdC}(\omega_c)}(\omega_{c})_{c'}\otimes (\omega_c)^{c'}\otimes \omega^{c}=\sum_{c\in E\text{AdC}(\omega)}\sum_{c'\in E\text{AdC}(\omega^c)}\omega_{c}\otimes (\omega^{c})_{c'}\otimes (\omega^{c})^{c'}.
  \end{equation}
 Consider first terms from the left-hand side of the above, i.e. of the form
  \begin{equation}\label{LeftTerm}
  (\omega_{c})_{c'}\otimes (\omega_c)^{c'}\otimes \omega^{c}
  \end{equation}
with $c\in E\text{AdC}(\omega)$ and $c'\in E\text{AdC}(\omega_c)$.
Since $c'\in E\text{AdC}(\omega_c)$ and $c\in E\text{AdC}(\omega)$ then $c'\in E\text{AdC}(\omega)$ by the first result of Proposition~\ref{LemmaHopf}. Furthermore $c'\in E\text{AdC}(\omega_c)$ implies that cuts $c$ and $c'$ are vertex-disjoint since $c'$ is cut-out of the remainder of $c$. Then $k:=c\cup c'$ is an extended admissible cut of $\omega$ and, by construction of $k$, $c$ is an extended admissible cut of $k$. Hence any term of the form given by Eq.~(\ref{LeftTerm}) is also of the form 
\begin{equation*}
  \omega_{k}\otimes (\omega^k)_{c}\otimes \omega^{c}
  \end{equation*}
with $\omega^k\in E\text{AdC}(\omega)$ and $\omega^c\in E\text{AdC}(\omega^k)$. This implies that the LHS of Eq.~(\ref{HopfProofEq}) is comprised in its RHS. Observe that this result is not true for admissible cuts, indeed we used that $k:=c\cup c'$ is the union of two non-overlapping cuts and so while $k\in E\text{AdC}(\omega)$, we have $k\notin \text{AdC}(\omega)$. This explains why $\Delta_{\text{CP}}$ fails to be coassociative. 

Second, consider terms from the RHS of Eq.~(\ref{HopfProofEq}),
  \begin{equation}\label{RightTerm}
  \omega_{c}\otimes (\omega^{c})_{c'}\otimes (\omega^{c})^{c'},
  \end{equation}
  with $c\in E\text{AdC}(\omega)$ and $c'\in E\text{AdC}(\omega^c)$. Since $c\in E\text{AdC}(\omega)$ and $c'\in E\text{AdC}(\omega^c)$ then $c'\in EAdC(\omega)$ by the second result of Proposition~\ref{LemmaHopf}. 
    Since $c'\in E\text{AdC}(\omega^c)$, $c'$ is entirely included within cut $c$ and we can define $l:=c\backslash c'$, $c=l\cup c'$ to be the extended admissible cut $\omega^l\in E\text{AdC}(\omega_{c'})$ which cuts out $c$ from the remainder $\omega_{c'}$. By construction $(\omega_{c'})^l=(\omega^c)_{c'}$ and $\omega_l=\omega_{c,c'}$. Consequently, any term of the form given by Eq.~(\ref{RightTerm}) is also of the form 
$$
(\omega_{c'})_l\otimes (\omega_{c'})^l\otimes \omega^{c'}
$$
with $\omega^{c'}\in E\text{AdC}(\omega)$ and $\omega^l\in E\text{AdC}(\omega_{c'})$. This implies that the RHS of Eq.~(\ref{HopfProofEq}) is comprised in its LHS. Remark that this statement is still true had we allowed only for admissible cuts. This is because if $c'\in \text{AdC}(\omega)$ and $c\in \text{AdC}(\omega_{c'})$ then $l:=c\backslash c'$ is an admissible cut of $\text{AdC}(\omega_{c'})$ by Proposition~\ref{LemmaCut}. This indicates that all terms generated by $\big(\text{Id}\otimes \Delta_{\text{CP}}\big)\circ\Delta_{\text{CP}}$ can be found in those generated by $\big( \Delta_{\text{CP}}\otimes \text{Id}\big)\circ\Delta_{\text{CP}}$, see e.g. Example~\ref{CopreLieEx}.~\\[-.5em]

The equality of Eq.~(\ref{HopfProofEq}) is proven and $\Delta_{\text{H}}$ is coassociative.
\end{proof}

 \begin{example}
    Let $\omega=1233234441$ be the walk of Examples~\ref{DeltaOmega}, \ref{CopreLieEx} and \ref{EAdCExample}. Then
 \begin{align*}
 \Delta_{\text{H}}(\omega)&=\mathbf{1}\otimes \omega + \omega \otimes \mathbf{1} +123323441\otimes 44 + 12332341 \otimes 444 + 123234441 \otimes 33 + 1234441\otimes 2332\\
&~+12323441\otimes 33\,|\,44+ 1232341\otimes 33\,|\,444+ 123441\otimes 2332\,|\,44+ 12341\otimes 2332\,|\,444.
 \end{align*}
 Omitting all terms involving $\mathbf{1}$ for the sake of concision and because they trivially satisfy the theorem as shown in its proof, we have
 \begin{align*}
    &(\Delta_{\text{H}}\otimes \Id)\circ \Delta_{\text{H}}(\omega) =\\
    &\hspace{7mm}12332341\otimes 44\otimes 44+12323441\otimes 33\otimes 44+123441\otimes 2332\otimes 44+1232341\otimes 33\,|\,44\otimes 44\\
    &\hspace{15mm}+12341\otimes 2332\,|\,44\otimes 44\\
    &\hspace{5mm}+1232341\otimes 33 \otimes 444 + 12341 \otimes 2332 \otimes 444\\
    &\hspace{5mm}+12323441\otimes 44 \otimes 33 +1232341\otimes 444 \otimes 33 +1234441 \otimes 232\otimes 33 +123441\otimes 232\,|\,44 \otimes 33\\
    &\hspace{15mm} +12341\otimes 232\,|\,444 \otimes 33\\
    &\hspace{5mm}+12341\otimes 444\otimes 2332+123441\otimes 44\otimes 2332\\
&\hspace{5mm}+12341\otimes 232\,|\,44\otimes 33\,|\,44+123441\otimes 232\otimes 33\,|\,44+1232341\otimes 44\otimes 33\,|\,44\\ 
&\hspace{5mm}+ 12341\otimes 232\otimes 33\,|\,444\\
&\hspace{5mm}+ 12341\otimes 44\otimes 2332\,|\,44.
 \end{align*}
 The presentation has been organized for the sake of readability: each line above represents terms steming from the same term found in $\Delta_{\mathrm{H}}(\omega)$, while an additional indentation denotes a continuing line. Similarly,
 \begin{align*}
&(\Id \otimes \Delta_{\text{H}})\circ \Delta_{\text{H}}(\omega) =12332341\otimes 44\otimes 44\\
&\hspace{5mm}+ 1234441\otimes   232\otimes 33 \\ 
&\hspace{5mm}+ 12323441\otimes 44\otimes 33 + 12323441\otimes 33\otimes 44\\
&\hspace{5mm}+1232341\otimes 33 \otimes 444
+1232341\otimes 444\otimes 33
+1232341\otimes 33\,|\,44\otimes 44 
+1232341\otimes 44\otimes 33\,|\,44 \\
&\hspace{5mm}+123441\otimes 2332\otimes 44 
+123441\otimes 232\,|\,44\otimes 33
+123441\otimes 232\otimes 33\,|\,44
+123441\otimes 44\otimes 2332\\
&\hspace{5mm}+12341\otimes 232\,|\,444\otimes 33
+12341\otimes 2332\,|\,44\otimes 44
+12341\otimes 232\,|\,44\otimes 33\,|\,44
+12341\otimes 2332\otimes 444\\
&\hspace{15mm}+12341\otimes 444\otimes 2332
+12341\otimes 232\otimes 33\,|\,444
+12341\otimes2332\,|\,44.
  \end{align*}
  A \emph{close} examination of both results reveals their equality as predicted by Theorem~\ref{HopfT}. 
\end{example}

	\begin{proposition}\label{HopfbiIdeal}
		Let $\G$ be a finite connected non-empty graph. We denote by $\mathcal{I}$ the vector space spanned by the elements $\omega_1|\dots|\omega_n-\omega_{\sigma(1)}|\dots|\omega_{\sigma(n)}$ where $\omega_1|\dots|\omega_n\in\Tens{\WGamma{W}{\G}}$ and $\sigma$ is a permutation. 
		Then, $\mathcal{I}$ is a Hopf bi-ideal of $\Tens{\WGamma{W}{\G}}$.
	\end{proposition}

	\begin{proof}
		By direct calculation, $\mathcal{I}$ is an ideal. 
		For the sake of brevity we denote by $E\text{AdC}_+(\omega)$ the set $E\text{AdC}(\omega)\cup\{\mathbf{1},\omega\}$, $\omega\in\mathcal{W}(\G)$. Let $\omega_1|\dots|\omega_n\in\Tens{\WGamma{W}{\G}}$ and $\sigma$ be a permutation on $n$ elements. Then,
		\begin{align*}
			&\hspace{-5mm}\Delta_{\text{H}}(\omega_1|\dots|\omega_n-\omega_{\sigma(1)}|\dots|\omega_{\sigma(n)})\\
   =&\displaystyle\sum_{c_i\in E\text{AdC}_+(\omega_i)}(\omega_{c_1}|\dots|\omega_{c_n})\otimes (\omega^{c_1}|\dots|\omega^{c_n})-\displaystyle\sum_{c_i\in E\text{AdC}_+(\omega_i)}(\omega_{c_\sigma(1)}|\dots|\omega_{c_\sigma(n)})\otimes (\omega^{c_\sigma(1)}|\dots|\omega^{c_\sigma(n)})\\
			=&\sum_{c_i\in E\text{AdC}_+(\omega_i)}(\omega_{c_1}|\dots|\omega_{c_n})\otimes (\omega^{c_1}|\dots|\omega^{c_n})-\displaystyle\sum_{c_i\in E\text{AdC}_+(\omega_i)}(\omega_{c_1}|\dots|\omega_{c_n})\otimes (\omega^{c_\sigma(1)}|\dots|\omega^{c_\sigma(n)})\\
			&+\sum_{c_i\in E\text{AdC}_+(\omega_i)}(\omega_{c_1}|\dots|\omega_{c_n})\otimes (\omega^{c_\sigma(1)}|\dots|\omega^{c_\sigma(n)})-\displaystyle\sum_{c_i\in E\text{AdC}_+(\omega_i)}(\omega_{c_\sigma(1)}|\dots|\omega_{c_\sigma(n)})\otimes (\omega^{c_\sigma(1)}|\dots|\omega^{c_\sigma(n)})
		\end{align*}
		This shows that \[\Delta_{\text{H}}(\omega_1|\dots|\omega_n-\omega_{\sigma(1)}|\dots|\omega_{\sigma(n)})\in\Tens{\WGamma{W}{\G}}\otimes \mathcal{I}+\mathcal{I}\otimes \Tens{\WGamma{W}{\G}},\] that is $\mathcal{I}$ is a co-ideal.
	\end{proof}

 Let $\G$ be a digraph. We define 
  \[
  \Sym{\WGamma{W}{\G}}:=\displaystyle\frac{\Tens{\WGamma{W}{\G}}}{\mathcal{I}}.\]
  In the vector space  $\Sym{\WGamma{W}{\G}}$ the concatenation product $\bigcdot$ becomes the disjoint-union product $\Box$. It follows from Theorem~\ref{HopfT} and Proposition~\ref{HopfbiIdeal} that:
  
	\begin{corollary}
		Let $\G$ be a digraph. Then $\mathcal{H}_{\mathcal{S}}:=(\Sym{\WGamma{W}{\G}},\Box,\Delta_{\emph{\text{H}}})$ is a Hopf algebra.
	\end{corollary}

	\subsection{Antipode}
	The existence of antipode maps in $\mathcal{H}_{\mathcal{T}}$ and $\mathcal{H}_{\mathcal{S}}$ is guaranteed by the fact that these are graded connected bialgebras. In this section we construct the antipodes explicitly, relying on the total order on decorated trees introduced by Foissy \cite{Foissy2002a}, which can be used in the present context thanks to Theorem~\ref{WalkTree}.

	
	\begin{definition}
	Let $\omega$ be a walk, $\text{AdC}(\omega)$ be its set of admissible cuts which we assume to be not empty. Let $1\leq n\leq |\text{AdC}(\omega)|$  be a positive integer, $c_i\in \text{AdC}(\omega)$ a collection of $n$  totally ordered, distinct, non-overlapping admissible cuts of $\omega$ with $c_1\leqslant_{\clk}\dots\leqslant_{\clk}  c_n$. Let $e:=c_1\,|\,\dots\,|\,c_{n}\in E\text{AdC}(\omega)$, we may also conveniently use the notation $|e|:=n$. 
      We associate to $e$ a tensor $T_{e}$ and a disjoint union $S_{e}$ as follows,
	\[
 T_{e}:=\omega_{c_1,\dots, c_n}\otimes (\omega_{c_1,\dots, c_{n-1}})^{c_n}\otimes\cdots \otimes (\omega_{c_1,\dots, c_{i-1}})^{c_i}\otimes \cdots\otimes (\omega_{c_1})^{c_2} \otimes \omega^{c_1},
 \]
	and 		
	\[
 S_{e}:=\omega_{c_1,\dots, c_n}\,\Box\, (\omega_{c_1,\dots, c_{n-1}})^{c_n}\,\Box\,\cdots \,\Box\, (\omega_{c_1,\dots, c_{i-1}})^{c_i}\,\Box\, \cdots\,\Box\, (\omega_{c_1})^{c_2} \,\Box\, \omega^{c_1}.
 \]
\end{definition}

	\begin{example}
		Consider again the walk of Example~\ref{CutOrderEx}, 
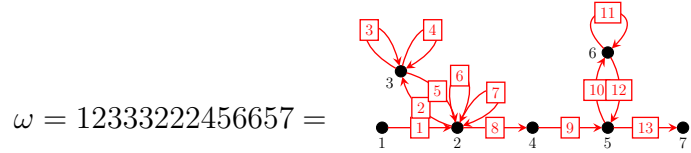
  \[\omega=12333222456657=\WalkExampleOrder{}{{1,...,7}}\] 
  and three of its admissible cuts $c_1=\omega^{2,4}$, $c_2=\omega^{3,4}$ and $c_3=\omega^{10,11}$. Since  $\omega^{3,4}\leqslant_{\clk} \omega^{2,4}\leqslant_{\clk} \omega^{10,11}$, for $e:=c_1\,|\, c_2\,|\,c_3\in E\text{AdC}(\omega)$, $|e|=3$,
		\[
  T_{e}=\WalkExampleOrder{cut}{{1,...,7}}\,|\hspace{-5.5mm}\ClosedWalkCorolla{1}{6}{{11}}\hspace{-4mm}|\hspace{-5.5mm}\ClosedWalkCorolla{1}{3}{{3}}\hspace{-4mm}|\hspace{-5.5mm}\ClosedWalkCorolla{1}{3}{{4}}
  \]
	\end{example}
~\\[-.5em]

	\begin{theorem}\label{TheoremAntipode}
		Let $\G$ be a digraph and $\omega\in\mathcal{W}(G)$. Then, in $\Tens{\WGamma{W}{\G}}$, the antipode $S(\omega)$ calculated on $\omega$ is,
		\begin{align*}S(\omega)&=-\omega - \sum_{e\in E\mathrm{AdC}(\omega)} (-1)^{|e|}T_e=-\omega-\sum_{n=1}^{|\mathrm{AdC}(\omega)|}\sum_{\substack{c_1\leqslant_{\clk}\dots \leqslant_{\clk}c_n\\ c_i\in \mathrm{AdC}(\omega)} }(-1)^{n}\,T_{c_1|\dots|c_n}\,
  \end{align*}
	where $|\mathrm{AdC}(\omega)|$ designates the cardinality of $\mathrm{AdC}(\omega)$.
	\end{theorem}

 	\begin{corollary}\label{AntipodeSym}
  Let $\G$ be a digraph and $\omega\in\mathcal{W}(G)$. Then, in $\Sym{\WGamma{W}{\G}}$, the antipode $S(\omega)$ calculated on $\omega$ is,
	\[S(\omega)=-\omega - \sum_{e\in E\mathrm{AdC}(\omega)} (-1)^{|e|}S_e=-\omega-\sum_{n=1}^{|\mathrm{AdC}(\omega)|}\sum_{\substack{c_1\leqslant_{\clk}\dots \leqslant_{\clk}c_n\\c_i\in \mathrm{AdC}(\omega)}}(-1)^{n}\,S_{c_1|\dots| c_n}\,,\]
	where $|\mathrm{AdC}(\omega)|$ designates the cardinality of $\mathrm{AdC}(\omega)$.
\end{corollary}	

	
	\begin{proof}[Proof of Theorem~\ref{TheoremAntipode}]
		We prove the theorem by induction on the cardinality of $\text{AdC}(\omega)$, using the relation $\varepsilon=\bigcdot\circ(\Id\otimes S)\circ \Delta_{\text{H}}$ where $\varepsilon$ is the counity of the Hopf algebra $\Tens{\WGamma{W}{\G}}$, and the algebra antimorphism relation $S(\omega|\omega')=S(\omega)S(\omega')$ for $\omega,\omega'\in\mathcal{W}(G)$.\\[-.5em]

  Firstly, if $\text{AdC}(\omega)=\emptyset$ then $\omega\in \text{SAW}(G)\cup \text{SAP}(G)$ and therefore $S(\omega)=-\omega$.\\[-.5em]

  Secondly, if $\text{AdC}(\omega)=\{\omega^{k,k'}\}$ then by Proposition~\ref{LemmaCut}, $\omega_{k,k'}\in \text{SAW}(G)\cup \text{SAP}(G)$ and 
  \[
  \Delta_{\text{H}}(\omega)=\omega\otimes \mathbf{1}+\mathbf{1}\otimes\omega+\omega_{k,k'}\otimes \omega^{k,k'}.\] Consequently, 
  \[
  S(\omega)=-\omega+\omega_{k,k'}\otimes \omega^{k,k'},
  \] 
  as claimed by the theorem.\\[-.5em]

  Thirdly, we assume that there exists and integer $n\in\NN$ such that the theorem is satisfied by any walk $\omega'\in\mathcal{W}(G)$ with $|\text{AdC}(\omega')|\leq n$. Consider $\omega\in\mathcal{W}(G)$ a walk with $|\text{AdC}(\omega)|=n+1$. Then,
				\begin{align*}
				S(\omega)=&-\omega-\sum_{\substack{k_1<k_1'<\dots <k_n<k_n'\\ \omega^{k_i,k_i'}\in E\text{AdC}(\omega)}}\omega_{k_1,k_1';\dots;k_n,k_n'}\bigcdot S(\omega^{k_1,k_1'}\bigcdot\dots\bigcdot\omega^{k_n,k_n'})\\
				=&-\omega-\sum_{\substack{k_1<k_1'<\dots <k_n<k_n'\\ \omega^{k_i,k_i'}\in E\text{AdC}(\omega)}}\omega_{k_1,k_1';\dots;k_n,k_n'}\bigcdot S(\omega^{k_s,k_s'})\bigcdot\dots\bigcdot S(\omega^{k_1,k_1'}).
			\end{align*}
			Thanks to Proposition~\ref{LemmaCut}, 
   \[
   \bigcup_{i=1}^{n}\text{AdC}(\omega^{k_i,k_i'})\subset \text{AdC}(\omega).
   \]
			and as a consequence, $\forall i\in\{1,\dots,n\}$, $|\text{AdC}(\omega^{k_i,k_i'})|\leq n$ and by induction hypothesis the theorem holds true for all $\omega^{k_i,k_i'}$. In particular, since any collection of $m$ admissible cuts of any $\omega^{k_i,k_i'}$ is totally ordered by $\leqslant_{\clk}$ per Proposition~\ref{TotalOrder},
			\[S(\omega)=-\omega-\sum_{n=1}^{|\text{AdC}(\omega)|}\sum_{\substack{\omega^{k_1,k_1'}\leqslant_{\clk}\dots \leqslant_{\clk} \omega^{k_n,k_n'}\\\omega^{k_i,k_i'}\in \text{AdC}(\omega)} }(-1)^{n}T_{\omega^{k_1,k_1'}\,|\,\dots\,|\, \omega^{k_n, k_n'}}.\]
	\end{proof}
	\begin{example}
		Consider the walk $$\omega=12223445=\WalkExampleAntipode{1}$$ 
		which has three admissible cuts
		$\text{AdC}(\omega)=\{\omega^{2,3},\omega^{1,3},\omega^{5,6}\}$ with
		$
		\omega^{2,3}\leqslant_{\clk}\omega^{1,3}\leqslant_{\clk}\omega^{5,6}.
		$
		Then the antipode of $\omega$ is 
		\begin{align*}
			S(\omega)=&-\omega+\WalkExampleAntipode{2}~|\hspace{-5.5mm}\ClosedWalkCorolla{1}{2}{{3}}\hspace{-4mm}+~ \WalkExampleAntipode{4}~|\hspace{-5.5mm}\ClosedWalkCorolla{1}{4}{{6}}\\[-.75em]
			&+~\WalkExampleAntipode{3}~|\hspace{-5mm}\ClosedWalkCorolla{2}{2}{{2,3}}\hspace{-3mm}-~\WalkSelfAvoiding{{1,...,5}}{{1,4,5,7}}{{}}~|\hspace{-5.5mm}\ClosedWalkCorolla{1}{4}{{6}}\hspace{-2mm}|\hspace{-5mm}\ClosedWalkCorolla{2}{2}{{2,3}}\\[-.5em]
			&-~\WalkExampleAntipode{3}~|\hspace{-5.5mm}\ClosedWalkCorolla{1}{2}{{2}}\hspace{-4mm}|\hspace{-5.5mm}\ClosedWalkCorolla{1}{2}{{3}}
			-~\WalkExampleAntipode{5}~|\hspace{-5.5mm}\ClosedWalkCorolla{1}{4}{{6}}\hspace{-4mm}|\hspace{-5.5mm}\ClosedWalkCorolla{1}{2}{{3}}\\[-.5em]
			&+\WalkSelfAvoiding{{1,...,5}}{{1,4,5,7}}{{}}|\hspace{-5.5mm}\ClosedWalkCorolla{1}{4}{{6}}\hspace{-4mm}|\hspace{-5.5mm}\ClosedWalkCorolla{1}{2}{{2}}\hspace{-4mm}|\hspace{-5.5mm}\ClosedWalkCorolla{1}{2}{{3}}.
		\end{align*}
	\end{example}

\section{Brace coalgebra and codendriform bialgebra on walks}\label{BraceSection}
\subsection{Brace coalgebra}
We show in this section that by paying attention to the number of admissible cuts appearing simultaneously in extended admissible cuts, we may endow $\mathcal{T}\langle \mathcal{W}(G)\rangle$ with a brace coalgebra structure from which the preLie co-structure on $\mathcal{W}(G)$ is recovered. 
We begin by recalling the necessary definitions pertaining to brace coalgebras.

\begin{definition}[$B_\infty$-algebra]
Let $\mathcal{V}$ be a vector space, $\mathcal{T}\langle\mathcal{V}\rangle$ the tensor algebra generated by $\mathcal{V}$ and let $\pi$ be the canonical projection from $\mathcal{T}\langle\mathcal{V}\rangle$ to $\mathcal{V}$. A $B_\infty$-algebra is a family $(\mathcal{V},(\langle-,-\rangle_{k,l})_{k,l\geq 0})$ where $\mathcal{V}$ is a vector space and for any $k,l\geq 0$,
$\langle-,-\rangle_{k,l}:\mathcal{V}^{\otimes k}\otimes \mathcal{V}^{\otimes l}\longrightarrow \mathcal{V}$ such that:
\begin{itemize}
\item[i)] $\langle-,-\rangle_{k,0}=\langle-,-\rangle_{0,k}=0$ if $k\neq 1$ and $\langle-,-\rangle_{1,0}=\langle-,-\rangle_{0,1}=\Id_\mathcal{V}$.
\item[ii)] The unique coalgebra morphism $m:\mathcal{T}\langle\mathcal{V}\rangle\otimes \mathcal{T}\langle\mathcal{V}\rangle\longrightarrow \mathcal{T}\langle\mathcal{V}\rangle$
defined by $\pi\circ m_{\mathcal{V}^{\otimes k}\otimes \mathcal{V}^{\otimes l}}=\langle-,-\rangle_k$ is associative. 
\end{itemize}
Then, equipped with the deconcatenation coproduct $\Delta_{\text{dec}}(v_1\ldots v_n):=\sum_{i=0}^n v_1\ldots v_i\otimes v_{i+1}\ldots v_n$, $(\mathcal{T}\langle\mathcal{V}\rangle,m,\Delta_{\text{dec}})$ is a Hopf algebra.
\end{definition}

A \textit{brace algebra} is a $B_\infty$-algebra such that $\langle-,-\rangle_{k,l}=0$ if $k\geq 2$. If $\mathcal{V}$ is a brace algebra, then for any $u=x_1\ldots x_k \in \mathcal{V}^{\otimes k}$ and $v\in \mathcal{T}\langle\mathcal{V}\rangle_+$,
\[m(u\otimes v)=\sum_{v\,=\,v_0\ldots v_{2k}} v_0 \langle x_1,v_1\rangle v_2\ldots \langle x_k, v_{2k-1}\rangle v_{2k},\]
where $v_i$ may be empty.\\[-.5em]

Dually, a locally finite brace coalgebra is a family $(\mathcal{V},(\delta_n)_{n\geq 1})$ where $\mathcal{V}$ is a vector space and for any $n$, $\delta_n:\mathcal{V}\longrightarrow \mathcal{V}\otimes \mathcal{V}^{\otimes n}$
such that $\delta_1=\Id_\mathcal{V}$; for any $v\in \mathcal{V}$, there exists $N(v) \in \mathbb{N}$ such that if $n\geq N(v)$, then $\delta_n(v)=0$; and the algebra morphism defined by
\begin{align}\label{BraceCoProduct}
\Delta :\left\{\begin{array}{rcl} \mathcal{T}\langle\mathcal{V}\rangle &\longrightarrow &\mathcal{T}\langle\mathcal{V}\rangle\otimes \mathcal{T}\langle\mathcal{V}\rangle\\
			v &\longmapsto& \displaystyle\Delta(v)=1\otimes v+v\otimes 1+\sum_{n\geq 1} ~\underbrace{~\delta_n(v)~}_{\in \mathcal{V}\otimes \mathcal{V}^{\otimes n}\subseteq \mathcal{T}\langle\mathcal{V}\rangle\otimes \mathcal{T}\langle\mathcal{V}\rangle},
		\end{array}\right.
\end{align}
is coassociative. Then, $(\mathcal{T}\langle\mathcal{V}\rangle,\bigcdot,\Delta)$ is a bialgebra, $\bigcdot$ being the concatenation product.

\begin{proposition}\label{BraceCoPreLie}
Let $(\mathcal{V},(\delta_n)_{n\geq 0})$ be a brace coalgebra with $\Delta$ the associated coassociative coproduct on $\mathcal{T}\langle\mathcal{V}\rangle$ as defined in Eq.~(\ref{BraceCoProduct}). Then $(\mathcal{V},\delta_1)$ is a preLie coalgebra. 
\end{proposition}
\begin{proof}
For any $v_1,\ldots,v_n \in \mathcal{V}$,
\[
(\pi\otimes \pi)\circ \Delta(v_1\ldots v_n)=\begin{cases}
\delta_1(v_1),&\text{if }n=1,\\
v_1\otimes v_2+v_2\otimes v_1,&\text{if }n=2,\\
0,&\text{otherwise}.
\end{cases}
\]
Therefore, for any $v\in \mathcal{V}$,
\begin{align*}
(\pi\otimes \pi\otimes \pi)\circ (\Delta \otimes \Id)\circ \Delta(v)&=(\pi\otimes \pi\otimes \Id)\circ (\Delta \otimes \Id)\circ \delta_1(v)\\
&=(\delta_1\otimes \Id)\circ \delta_1(v),\\
(\pi\otimes \pi\otimes \pi)\circ (\Id \otimes \Delta)\circ \Delta(v)&=\sum_{k=1}^\infty (\Id \otimes \pi\otimes \pi)\circ (\Id \otimes \Delta)\circ \delta_n(v)\\
&=(\Id \otimes \delta_1)\circ \delta_1(v)+(\Id \otimes \Id \otimes \Id+\Id \otimes \tau)\circ \delta_2(v).
\end{align*} 
As a consequence, by the coassociativity of $\Delta$,
\[(\delta_1\otimes \Id)\circ \delta_1-(\Id \otimes \delta_1)\circ \delta_1=(\Id \otimes \Id \otimes \Id+\Id \otimes \tau)\circ \delta_2,\]
and it follows that,
\[(\delta_1\otimes \Id)\circ \delta_1-(\Id \otimes \delta_1)\circ \delta_1=(\Id \otimes \Id \otimes \Id+\Id \otimes \tau)\circ((\delta_1\otimes \Id)\circ \delta_1-\circ (\Id \otimes \delta_1)\circ \delta_1),\]
so $(\mathcal{V},\delta_1)$ is a preLie coalgebra. 
\end{proof}

In the case of interest here, namely that of $\mathcal{W}(G)$, define 
for $\omega\in\mathcal{W}(G)$,
\[\delta_n(\omega):=\sum_{c\in E_n\text{AdC}(w)}
\omega_{c}\otimes \omega^{c},\]
with $c$ an extended admissible cut \textit{involving exactly $n$ admissible cuts}, i.e. $c\in E_n\text{AdC}(w)\iff\omega^{c}=\omega^{k_1,k_1'}\,|\,\cdots\,|\,\omega^{k_n,k_n'}$, $\omega^{k_i,k_i'}\in\text{AdC}(\omega)$.
By Theorem~\ref{HopfT}, the coproduct defined as in Eq.~(\ref{BraceCoProduct}) with the above definition for the $\delta_n$, namely $\Delta_\mathrm{H}$, is coassociative.  Proposition~\ref{BraceCoPreLie} then implies that $(\mathcal{W}(G),\delta_1)$, $\delta_1\equiv \Delta_{\text{CP}}$, is a preLie coalgebra. In other terms, Theorem~\ref{DeltaCPTHM} may be seen as a corollary of Theorem~\ref{HopfT}.

\subsection{Codendriform bialgebra}
The brace co-structure on $\mathcal{W}(G)$ now implies that $\mathcal{T}(\mathcal{W}(G))$ is a codendriform bialgebra, a dual of the results of \cite{Ronco2001}. 

Denoting by $E\text{AdC}_+(\omega)=E\text{AdC}(\omega)\cup\{\mathbf{1},\omega\}$ the set of extended admissible cuts of $\omega$ augmented by the empty cut and the total cut, recall that for any $n\geq 1$ walks $\omega_1,\cdots,\omega_n\in\mathcal{W}(G)$,
\[
\Delta_{\mathrm{H}}(w_1\mid \ldots \mid w_n)=\sum_{c_i \in EAdC_+(\omega_i)} (\omega_1)_{c_1}\mid \ldots \mid (\omega_n)_{c_n}\otimes \omega_1^{c_1}\mid \ldots \mid \omega_n^{c_n}.
\]
Now define, for any nonempty word $\omega_1\mid \ldots \mid \omega_n\in\mathcal{T}\langle \mathcal{W}(G)\rangle$, the maps
\begin{align*}
\Delta_\prec(\omega_1\mid \ldots \mid \omega_n)&:=\sum_{\substack{c_i \in E\text{AdC}_+(\omega_i),\\ (\omega_1)_{c_1}\neq \mathbf{1}}} (\omega_1)_{c_1}\mid \ldots \mid (\omega_n)_{c_n}\otimes \omega_1^{c_1}\mid \ldots \mid \omega_n^{c_n},\\
\Delta_\succ(w_1\mid \ldots \mid w_n)&:=\sum_{\substack{c_i \in E\text{AdC}_+(w_i),\\ (\omega_1)_{c_1}=\mathbf{1}}} (\omega_1)_{c_1}\mid \ldots \mid (\omega_n)_{c_n}\otimes \omega_1^{c_1}\mid \ldots \mid \omega_n^{c_n}.
\end{align*}
\begin{proposition}
    $(\mathcal{T}(\mathcal{W}(G)),\Delta_\prec,\Delta_\succ)$ is a codendriform bialgebra. Furthermore, for any $x\in \mathcal{T}(\mathcal{W}(G))$ with no constant term and any $y\in \mathcal{T}(\mathcal{W}(G))$, $\Delta_\prec(x\mid y)=\Delta_\prec(x)\mid \Delta_\mathrm{H}(y)$, and $\Delta_\succ(x\mid y)=\Delta_\succ(x)\mid \Delta_\mathrm{H}(y)$.
\end{proposition}
\begin{proof}
By the coassociativity of $\Delta_\mathrm{H}$, for any $\omega\in\mathcal{W}(G)$, $(\Delta_{\mathrm{H}}\otimes \Id)\circ \Delta_{\mathrm{H}}(\omega)=(\Id \otimes \Delta_{\mathrm{H}})\circ \Delta_{\mathrm{H}}(\omega)$, that is
\begin{align*}
&\sum_{c\in E\text{AdC}_+(\omega)}\sum_{c'\in E\text{AdC}_+(\omega_c)}(\omega_{c})_{c'}\otimes (\omega_c)^{c'}\otimes \omega^{c}=\sum_{c\in E\text{AdC}_+(\omega)}\sum_{c'\in E\text{AdC}_+(\omega^c)}\omega_{c}\otimes (\omega^{c})_{c'}\otimes (\omega^{c})^{c'}.
\end{align*}
Thus, there exists a set $E\text{AdC}_+^{(2)}(\omega)$ such that the above may be put in the form
$$
\sum_{c\in E\text{AdC}_+^{(2)}(\omega)} \omega_{c}\otimes \omega^{c(1)}\otimes \omega^{c(2)}.
$$
Then, using this notation,
\begin{align*}
&\hspace{-15mm}(\Delta_{\mathrm{H}}\otimes \Id)\circ \Delta_\succ(\omega_1|\ldots|\omega_n)=(\Id \otimes \Delta_\succ)\circ \Delta_\succ(\omega_1|\ldots|\omega_n)\\
&=\sum_{\substack{c_i \in EAdC_+^{(2)}(w_i),\\ (\omega_1)_{c_1}=\,\omega_1^{c_1(1)}=\mathbf{1}}} (\omega_1)_{c_1}\mid \ldots \mid (\omega_n)_{c_1}\otimes \omega_1^{c_1(1)}\mid \ldots \mid \omega_n^{c_n(1)}
\otimes \omega_1^{c_1(2)}\mid \ldots \mid \omega_n^{c_n(2)},\\
&\hspace{-15mm}(\Delta_\succ\otimes \Id)\circ \Delta_\prec(\omega_1|\ldots|\omega_n)=(\Id \otimes \Delta_\prec)\circ \Delta_\succ(\omega_1|\ldots|\omega_n)\\
&=\sum_{\substack{c_i \in E\text{AdC}_+^{(2)}(w_i),\\ (\omega_1)_{c_1}=\mathbf{1},\:\omega_1^{c_1(1)}\neq \mathbf{1}}} (\omega_1)_{c_1}\mid \ldots \mid (\omega_n)_{c_1}\otimes \omega_1^{c_1(1)}\mid \ldots \mid \omega_n^{c_n(1)}
\otimes \omega_1^{c_1(2)}\mid \ldots \mid \omega_n^{c_n(2)},\\
&\hspace{-15mm}(\Delta_\prec\otimes \Id)\circ \Delta_\prec(\omega_1|\ldots|\omega_n)=(\Id \otimes \Delta_{\mathrm{H}})\circ \Delta_\prec(\omega_1|\ldots|\omega_n)\\
&=\sum_{\substack{c_i \in E\text{AdC}_+^{(2)}(w_i),\\ (\omega_1)_{c_1}\neq \mathbf{1}}} (\omega_1)_{c_1}\mid \ldots \mid (\omega_n)_{c_1}\otimes \omega_1^{c_1(1)}\mid \ldots \mid \omega_n^{c_n(1)}
\otimes \omega_1^{c_1(2)}\mid \ldots \mid \omega_n^{c_n(2)}.
\end{align*}
\end{proof}

	\section{Cacti, towers and corollas}\label{CactiTowerCorollas}
Recall from Definition~\ref{DefCac} that a cactus is a kind of ``disentangled" walk  resembling a self-avoiding skeleton on which bouquets of towers are attached. Given that by the proof of Theorem~\ref{WalkTree} all walks are chronologically equivalent to cacti, it seems intuitive that bouquets and towers are basic building blocks of walks and ought to be associated to sub-algebras of the walk algebras. In this section we formalize this observation by showing first that cacti, towers and corollas (a special type of bouquets) give rise to sub-Hopf algebras of the tensor and symmetric algebras of all walks; and secondly that the mapping from walks to cacti effected by the map $C$ defined in the proof of Theorem~\ref{WalkTree} generates Hopf algebra morphisms.

In a later work, using the permutative non-associative product nesting and the NAP-copreLie operad it forms with $\Delta_{\text{CP}}$, we will formalize and exploit algebraically  the construction of walks from bouquets and towers based on Lawler's process.

\subsection{Hopf subalgebras associated to cacti, towers and corollas}
 \begin{definition}[Tower]\label{TowerDef}
		Let $\G$ be a digraph. A tower with root $r_1$ and of height $n\in\mathbb{N}\backslash\{0\}$ is a closed walk made of a collection $\text{Cycl}_1$, \dots, $\text{Cycl}_n$ of \emph{simple cycles} with roots $r_1$, \dots, $r_n$, respectively, and such that:
  \begin{enumerate}
			\item[i)] $V(\text{Cycl}_k)\cap V(\text{Cycl}_{k+1})=\{r_{k+1}\}$ for any $k\in\{1,\dots,n-1\}$, 
			\item[ii)] $V(\text{Cycl}_k)\cap V(\text{Cycl}_{l})=\emptyset$ whenever $|k-l|> 1$.
		\end{enumerate}
		The vector space spanned by the towers of $\G$ is denoted by $\text{Tow}(\G)$. The space $\Tens{\text{Tow}(\G)}$ (respectively $\Sym{\text{Tow}(\G)}$) is the tensor algebra (respectively the symmetric algebra) generated by $\text{Tow}(\G)$.
	\end{definition}
	
	
	\begin{definition}[Corolla]
		Let $\G$ be a digraph. A corolla of root $r$ in $\G$ is a closed walk made of $n\in\mathbb{N}\backslash\{0\}$ simple cycles $\text{Cycl}_1$, \dots, $\text{Cycl}_n$, all with a common root $r$. Corollas are bouquets of simple cycles.
        
		The vector space spanned by all corollas (respectively corollas of root $r$) of $\G$ is $\text{Cor}(\G)$ (respectively $\text{Cor}_r(\G)$). The space $\Tens{\text{Cor}(\G)}$ (respectively $\Sym{\text{Cor}(\G)}$) is the tensor algebra (respectively the symmetric algebra) generated by $\text{Cor}(\G)$.  We define the spaces $\Tens{\text{Cor}_r(\G)}$ and $\Sym{\text{Cor}_r(\G)}$ similarly from $\text{Cor}_r(\G)$.
	\end{definition}
	
	\begin{example}
 The walk $123454321\in\text{Tow}(G)$ is a tower, while walks $111\in\text{Cor}_1(G)$ and $123412451\in\text{Cor}_1(G)$ are corollas with root $1$.
	\end{example}


	\begin{proposition}\label{SubAlgebras}
		Let $\G$ be a digraph and $r\in V(\G)$. Then,
		\begin{enumerate}
			\item $(\Tens{\emph{\text{Tow}}(\G)},\bigcdot,\Delta_{\emph{\text{H}}})$,  $(\Tens{\emph{\text{Cor}}_r(\G)},\bigcdot,\Delta_{\emph{\text{H}}})$, $(\Tens{\emph{\text{Cor}}(\G)},\bigcdot,\Delta_{\emph{\text{H}}})$ and $(\Tens{\emph{\text{Cact}}(\G)},\bigcdot,\Delta_{\emph{\text{H}}})$ are Hopf subalgebras of $(\Tens{\WGamma{W}{\G}},\bigcdot,\Delta_{\emph{\text{H}}})$.
			\item $(\Sym{\emph{\text{Tow}}(\G)},\Box,\Delta_{\emph{\text{H}}})$, $(\Sym{\emph{\text{Cor}}_r(\G)},\Box,\Delta_{\emph{\text{H}}})$, $(\Sym{\emph{\text{Cor}}(\G)},\Box,\Delta_{\emph{\text{H}}})$  and $(\Sym{\emph{\text{Cact}}(\G)},\Box,\Delta_{\emph{\text{H}}})$ are Hopf subalgebras of $(\Sym{\WGamma{W}{\G}},\Box,\Delta_{\emph{\text{H}}})$.
		\end{enumerate}
	\end{proposition}
	
	\begin{proof}
  Firstly, the claims regarding $(\Tens{\text{Tow}(\G)},\bigcdot,\Delta_{\text{H}})$ and $(\Sym{\text{Tow}(\G)},\Box,\Delta_{\text{H}})$ are shown by direct calculation.
			
   Secondly, let $\omega$ be a corolla with root $r\in V(G)$ comprising $n\in\mathbb{N}
   \backslash\{0\}$ simple cycles $\text{Cycl}_{1\leq k\leq n}$. Let $v\in V(G)$ be a 
   vertex other than the root $r$ visited by $\omega$. Since $\text{Cycl}_1$, \dots, 
   $\text{Cycl}_n$ are simple cycles, if $v$ is visited several times by $\omega$ 
   then two instances of $v$ cannot be found within a unique simple cycle. But by 
   using Remark~\ref{AltLES} equivalent to Definition~\ref{def:admiscut} for the loop-
   erased sections, any subwalk $\omega^{l,l'}=w_l\cdots w_{l'}$ with $w_l=w_{l'}=v$ 
   is not an admissible cut of $\omega$, as it is not a valid loop-erased section of 
   $\omega$. Then all the admissible cuts of $\omega$ take place at the root $r$, \vspace{-5mm} 
			\[
   \Delta_{\text{H}}(\omega)=\omega \otimes 1+1\otimes \omega+\sum_{p=1}^{n-1}\GeneralClosedWalkCorolla{r}{\text{Cycl}_1}{\text{Cycl}_k} \otimes \GeneralClosedWalkCorolla{r}{\text{Cycl}_{k+1}}{\text{Cycl}_n}
   \]
			which implies the claims for $(\Tens{\text{Cor}(\G)},\bigcdot,\Delta_{\text{H}})$, $(\Sym{\text{Cor}_i(\G)},\Box,\Delta_{\text{H}})$ and $(\Sym{\text{Cor}(\G)},\Box,\Delta_{\text{H}})$.
	
 Thirdly, the claims about $(\Tens{\text{Cact}(\G)},\bigcdot,\Delta_{\text{H}})$ and $(\Sym{\text{Cact}(\G)},\Box,\Delta_{\text{H}})$ both follow from Proposition~\ref{LemmaHopf} and the fact that an admissible cut of a walk $\omega$ is, by definition, a loop-erased section of $\omega$.
	\end{proof}

\begin{remark}
Since Proposition~\ref{SubAlgebras} establishes Hopf algebra structures on the tensor algebras generated by towers, corollas and cacti, the constructions of \S\ref{BraceSection} extend to these walks as well. That is, there are brace coalgebras and codendriform bialgebras on  towers, corollas and cacti and these are sub coalgebras of the structures of \S\ref{BraceSection} on all walks.
\end{remark}

\subsection{The cactus map generates Hopf algebra morphisms}\label{Cmorphism}
We now show that the map $C$ defined in Eq.~(\ref{CactusMapDef}) which sends a walk $\omega$ to a cactus generates Hopf algebra morphisms. Recall that, by definition,  $C(\omega)$ is a cactus in the complete graph \(K_\mathbb{N}\) with $V(K_\mathbb{N})=\mathbb{N}$. 

	Let $\mathcal{I}_\mathbb{N}$ be the set of the injective maps  $\mathbb{N}\to \mathbb{N}$. For $f\in\mathcal{I}_\mathbb{N}$ and $\omega=w_0\cdots w_\ell\in\mathcal{W}(K_\mathbb{N})$, we denote by $f(\omega)\in\mathcal{W}(K_\mathbb{N})$ the walk defined by $f(\omega):=f(w_0)\dots f(w_\ell)$.
	
	\begin{definition}
		 Let $\mathcal{J}_1$ and $\mathcal{J}_2$ be the vector spaces defined by:
			\begin{align*}
				\mathcal{J}_1&:=\text{Span}\big(\omega_1\,|\,\dots\,|\,\omega_n-f_1(\omega_1)\,|\,\dots\,|\,f_n(\omega_n);~ n\in\NN\backslash\{0\},  \omega_i\in\text{Cact}(K_\mathbb{N}), f_i\in\mathcal{I}_{\mathbb{N}}\big),\\			
				\mathcal{J}_2&:=\text{Span}\big(\omega_1\,\Box\,\dots\,\Box\,\omega_n-f_1(\omega_1)\,\Box\,\dots\,\Box\,f_n(\omega_n);~ n\in\NN\backslash\{0\},  \omega_i\in\text{Cact}(K_\mathbb{N}), f_i\in\mathcal{I}_{\mathbb{N}}\big).
			\end{align*}  
	\end{definition}
	
	\begin{proposition}
		The vector space $\mathcal{J}_1$ (respectively $\mathcal{J}_2$)  is a Hopf biideal of $\Tens{\emph{\text{Cact}}(K_\mathbb{N})}$ (respectively $\Sym{\emph{\text{Cact}}(K_\mathbb{N})}$).
	\end{proposition}
	
	\begin{proof}
		We prove the result for $\mathcal{J}_1$. The reasoning for $\mathcal{J}_2$ is entirely similar.

  Let $\omega=w_0\cdots w_\ell\in \mathcal{W}(G)$, then for any injective map $f\in \mathcal{I}_{\mathbb{N}}$, the length of $f(\omega)$ is still $\ell$ and we have the relation Eq.~(\ref{finjectAdC}), that is
  \begin{equation}\label{finjectAdC2}
   \omega^{k,k'}\in \text{AdC}(\omega)\iff f(\omega)^{k,k'}\in \text{AdC}(f(\omega)).
   \end{equation} 
Therefore, if $\omega\in\text{Cact}(G)$, $f(\omega)$ is also a cactus.\\[-.5em]

			Now let $\alpha:=\omega_1|\dots|\omega_n-f_1(\omega_1)|\dots|f_n(\omega_n)$ be a generator of $\mathcal{J}_1$ and $\beta:=\tau_1|\dots|\tau_m\in\Tens{\text{Cact}(K_\mathbb{N})}$,
			\begin{align*}
				\alpha\bigcdot \beta&=\omega_1\,|\,\dots\,|\,\omega_n\,|\,\tau_1\,|\,\dots\,|\,\tau_m-f_1(\omega_1)\,|\,\dots\,|\,f_n(\omega_n)\,|\,\tau_1\,|\,\dots\,|\,\tau_m\\
				&=\omega_1\,|\,\dots\,|\,\omega_n\,|\,\tau_1\,|\,\dots\,|\,\tau_m-f_1(\omega_1)\,|\,\dots\,|\,f_n(\omega_n)\,|\,\Id(\tau_1)\,|\,\dots\,|\,\Id(\tau_m).
			\end{align*}
			So we obtain $\alpha \bigcdot \beta \in\mathcal{J}_1$ and similarly $\beta\bigcdot \alpha\in\mathcal{J}_1$. As a consequence, $\mathcal{J}$ is an ideal.\\[-.5em]

    Let  $\omega=w_0\dots w_\ell\in\text{Cact}(G)$ and $f\in\mathcal{I}_\mathbb{N}$. By injectivity of $f$ for $c\in E\text{AdC}(\omega)$ with $\omega^c:=\omega^{k_1,k_1';\dots;k_nk_n'}$, we have
    \[
   f(\omega)_c=f(\omega)_{k_1,k_1';\dots;k_nk_n'}=f(\omega_{k_1,k_1';\dots;k_nk_n'})=f(\omega_c).
   \]
   Therefore
			\begin{align*}
				\Delta_{\text{H}}(\omega-f(\omega))&=(\omega-f(\omega))\otimes \mathbf{1}+\mathbf{1}\otimes(\omega-f(\omega))
				+\displaystyle\sum_{c\in E\text{AdC}(\omega)}\big\{ \omega_{c}\otimes \omega^c-f(\omega)_{c}\otimes f(\omega)^{c}\big\},\\
    &=(\omega-f(\omega))\otimes \mathbf{1}+\mathbf{1}\otimes(\omega-f(\omega))
				+\displaystyle\sum_{c\in E\text{AdC}(\omega)}\big\{ \omega_{c}\otimes \omega^c-\omega_{c}\otimes f(\omega)^{c}\big\}\\
    &\hspace{65mm}+\displaystyle\sum_{c\in E\text{AdC}(\omega)}\big\{ \omega_{c}\otimes f(\omega)^{c}-f(\omega_{c})\otimes f(\omega)^c\big\}.
			\end{align*}
			This shows that $\Delta_{\text{H}}(\omega-f(\omega))\in\Tens{\text{Cact}(K_\mathbb{N})}\otimes\mathcal{J}_1+\mathcal{J}_1\otimes\Tens{\text{Cact}(K_\mathbb{N})}$.
			Since furthermore $\Delta_{\text{H}}$ is an algebra morphism, we conclude that $\mathcal{J}_1$ is a coideal.\\[-.5em]
			 
    Finally, by Eq.~(\ref{finjectAdC}), Theorem~\ref{TheoremAntipode} and the fact the antipode is an algebra antimorphism, we get $S(\mathcal{J}_1)\subset \mathcal{J}_1$.
	\end{proof}
	
	\begin{remark}
		The elements of $\Tens{\text{Cact}}(K_\mathbb{N})/\mathcal{J}_1$ and $\displaystyle\Sym{\text{Cact}}(K_\mathbb{N})/\mathcal{J}_2$ can be seen as cacti where the node labels have been forgotten since the node labels are defined modulo the action of $\mathcal{I}_\mathbb{N}$. These Hopf algebras can thus legitimately be called the tensor and symmetric Hopf algebras of unlabeled cacti, respectively. 
	\end{remark}
	
	

	By direct calculation,
	\begin{proposition}
		The degree map $\deg$ makes $\Tens{\emph{\text{Cact}}(K_\mathbb{N})}/\mathcal{J}_1$ and $\Sym{\emph{\text{Cact}}(K_\mathbb{N})}/\mathcal{J}_2$ into graded Hopf algebras.
	\end{proposition}
	
	\begin{theorem}\label{WalkToCactus}
		Let $\G$ be a digraph. Let $\Phi_1:\Tens{\WGamma{W}{\G}}\rightarrow \displaystyle\Tens{\emph{\text{Cact}}(K_\mathbb{N})}/\mathcal{J}_1$ and $\Phi_2:\Sym{\WGamma{W}{\G}}\rightarrow  \Tens{\emph{\text{Cact}}(K_\mathbb{N})}/\mathcal{J}_2$  be the two algebra morphisms such that $\Phi_i(\omega)$ is the unlabeled cactus obtained from $C(\omega)$ by forgetting all its node labels. 
		Then $\Phi_1$ and $\Phi_2$ are Hopf algebra morphisms.
	\end{theorem}
	
	\begin{proof}
		By definition, the cardinalities of  $V(\omega)$ and $V(C(\omega))$ are equal,  $C(\omega)$ is a cactus and Eq.~(\ref{finjectAdC}) holds. By the definition of $\Delta_{\text{H}}$ and the formulas of the antipode given in Theorem~\ref{TheoremAntipode} and Corollary~\ref{AntipodeSym}, we prove the theorem. 	
	\end{proof}

	\section{Acknowledgements}
	C. Mammez and P.-L. Giscard are supported by the ANR \textsc{Alcohol} project ANR-19-CE40-0006. In addition, C. Mammez aknowledges support from Labex CEMPI, ANR-11-LABX-0007-01. P.-L. Giscard also received funding from ANR \textsc{Magica} project ANR-20-CE29-0007.
 \bibliographystyle{siam}

\end{document}